\documentclass[11pt]{amsart}
\usepackage[margin=1in]{geometry}

\usepackage[english]{babel}

\usepackage[dvipsnames]{xcolor}
\usepackage{tikz}
\usetikzlibrary{fit, backgrounds}
\usetikzlibrary{shapes.geometric}
\usetikzlibrary{svg.path}
\usetikzlibrary{decorations.pathreplacing}

\usepackage{amsmath}
\usepackage{amssymb}
\usepackage{amsthm}
\usepackage{xcolor}
\usepackage{graphicx}
\usepackage[colorlinks=true, allcolors=blue]{hyperref}
\usepackage{cleveref}
\usepackage{enumitem}

\newtheorem{lemma}{Lemma}[section] 
\newtheorem{theorem}[lemma]{Theorem}
\newtheorem{corollary}[lemma]{Corollary}
\newtheorem{proposition} [lemma]{Proposition}
\newtheorem{observation}[lemma]{Observation}

\numberwithin{equation}{section}
\theoremstyle{definition}
\newtheorem{definition}[lemma] {Definition}
\newcounter{tbox}
\newcommand{\sta}[1]{\vspace*{0.3cm}\refstepcounter{tbox}\noindent{ \parbox{\textwidth}{(\thetbox) \emph{#1}}}\vspace*{0.3cm}}

\newtheorem*{thm:graph operations}{\Cref{Prop: Join superadditive} and \Cref{Theorem: Disjoint Union subadditive}}
\newtheorem*{thm:interval}{\Cref{Theorem: I_k is an interval}}
\newtheorem*{thm:kco graph}{\Cref{Theorem: every graph is a k-coalition partition graph}}

\newcommand{\defn}[1]{\textcolor{Maroon}{\emph{#1}}}
\newcommand{\N}{\mathbb{N}}
\newcommand{\ints}{\mathbb{Z}}

\DeclareMathOperator{\COk}{CO_k}
\DeclareMathOperator{\cok}{co_k}
\newcommand{\copart}{\mathfrak{X}}
\newcommand{\kcograph}{\operatorname{kCG}}
\newcommand{\kset}{\text{I}_k}

\title{On $k$-coalition partitions of graphs}
\author{Claire Kaneshiro}
\address{Department of Mathematics, Princeton University, Princeton, NJ}
\email{kaneshiroclaire@gmail.com}
\date{\today}

\begin{document}

\begin{abstract} 
In a graph, a set $D$ is $k$-dominating if every vertex in $V(G) \setminus D$ has at least $k$ neighbors in $D$. Jafari, Alikhani, and Bakhshesh  introduced the concept of a $k$-coalition, which is a pair of disjoint sets $X_1$ and $X_2$ of vertices such that neither is a $k$-dominating set but $X_1 \cup X_2$ is a $k$-dominating set. A $k$-coalition partition is a vertex partition in which each set either forms a $k$-coalition with some other set or is itself a $k$-dominating set with exactly $k$ vertices. The $k$-coalition number $\COk(G)$ is the maximum number of sets in a $k$-coalition partition. 
We compute the $k$-coalition number for several families and bound the $k$-coalition number under disjoint union and graph join. We show that the set of possible sizes of $k$-coalition partitions forms an interval. Finally, we investigate $k$-coalition graphs and prove that every graph is a $k$-coalition graph.
\end{abstract}
\maketitle

\section{Introduction}
Domination in graphs is a rich and well-studied area in graph theory. One generalization of a dominating set is a \textit{coalition}. The term ``coalition'' has been used in several contexts in graph theory, network theory, and game theory. In this paper, we adopt the notion of coalitions introduced in 2020 by Haynes, Hedetniemi, Hedetniemi, McRae, and Mohan in \cite{HHHMM20} and expanded on in several subsequent works \cite{FairCoalition25, EdgeCoalition25,BHS25, HHHMMBounds}.

Two disjoint sets of vertices in a graph form a \defn{coalition} if neither is a dominating set, but their union is a dominating set. A \defn{coalition partition} is a partition of the vertex set in which every set either (1) forms a coalition with another set or (2) is a dominating set with exactly one vertex. Much of the previous work has focused on the \defn{coalition number}, which is the maximum number of sets in a coalition partition of a graph $G$. 

While coalitions were initially based on standard (vertex) domination, recent work has generalized this concept to other domination types, for example, independent coalition \cite{IndependentCoalition}, connected coalitions \cite{ConnectedCoalition}, double coalition \cite{DoubleCoalition1,DoubleCoalition2}, and total $k$-coalition \cite{TotalKCoalition}, which correspond to independent domination, connected domination, $2$-domination, and total $k$-domination (also referred to as $k$-tuple total domination), respectively.

Recently, Jafari, Alikhani, and Bakhshesh \cite{JAB25} introduced $k$-coalitions based on $k$-domination. A set $S\subseteq V$ is a \defn{$k$-dominating set} if every vertex in $V \setminus S$ has at least $k$ neighbors in $S$. Two disjoint sets, $S_1$ and $S_2$, form a \defn{$k$-coalition} if neither is individually a $k$-dominating set but $S_1 \cup S_2$ is a $k$-dominating set. When $k=1$, this reduces to the (standard) definition of coalitions, so we assume $k \geq 2$ throughout this paper. 
Analogously, the authors define a \defn{$k$-coalition partition} as a partition of the vertex set in which every set either forms a $k$-coalition with at least one other set in the partition or is itself a $k$-dominating set with exactly $k$ vertices. The \defn{$k$-coalition number}, denoted $\COk(G)$, is the maximum number of sets over all $k$-coalition partitions of $G$. In \cite{JAB25}, the authors compute the $k$-coalition number (with an emphasis on $k=2$) for specific graph classes, including complete graphs, cycles, paths, trees, and some corona products. They also characterize graphs with $2$-coalition number equal to the number of vertices in the graph. 

In 2025, Bre\v{s}ar, Henning, and Samadi in \cite{BHS25} established general upper and lower bounds on the $k$-coalition number, including the sharp bound on the $2$-coalition number of a graph $G$
  $$CO_2(G) \leq \left(\Delta - 2\left\lfloor\frac{\delta}{2}\right\rfloor + 1\right)\left(\left\lfloor\frac{\delta}{2}\right\rfloor + 1\right)+\left\lfloor\frac{\delta}{2}\right\rfloor + 1$$
in terms of its minimum and maximum degree. The authors provide an upper bound for the $2$-coalition number of an $n$-vertex tree $T$ 
and characterize the trees attaining this bound. They compute the $k$-coalition number exactly for cubic graphs and complete bipartite graphs, fully resolving an open case from \cite{JAB25}.
In a related development, Bre\v{s}ar, Klav\v{z}ar, and Samadi \cite{TotalKCoalition} introduced the concept of total $k$-coalitions, corresponding to total $k$-domination. If, in the definition of $k$-domination, we replace $V \setminus S$ with $V$, we get the definition of total $k$-domination, which requires that every vertex in $V$ (including every vertex in $S$) has at least $k$ neighbors in the $k$-dominating set $S$.

Jafari, Alikhani, and Bakhshesh proposed several questions and directions for future research. We address many of these questions and highlight our main findings in the outline below. 
The remainder of the paper is organized as follows:
\begin{enumerate} [leftmargin = 2cm]
    \item[\textbf{\Cref{sect: definitions and notation}}] includes formal definitions and relevant background on $k$-coalitions. 

    \item[\textbf{\Cref{sect: Bounds on k-coalition number}}] provides useful bounds on $\COk(G)$ in terms of the minimum and maximum vertex degree. We apply these bounds to a few examples. In particular, we determine the $k$-coalition number of complete multipartite graphs, generalizing the work of \cite{JAB25} and \cite{BHS25}. Since multipartite graphs are the graph join of independent sets, this relates nicely to \Cref{sect: graph operations}.
    
    \item[\textbf{\Cref{sect: graph operations}}] characterizes the behavior of the $k$-coalition number under graph operations. Motivated by the study of graph invariants under graph sums, we provide the following two results, where the latter is much harder:
    \begin{thm:graph operations}
            The $k$-coalition number is superadditive under graph joins and strictly subadditive under disjoint unions. 
    \end{thm:graph operations}
    We also study possible extremal behavior of the join or the disjoint union of a graph with itself $n$ times. 
    
    \item[\textbf{\Cref{sect: I_k(G)} }] introduces the minimum $k$-coalition number, denoted by $\cok(G)$, which is the minimum number of sets in a $k$-coalition partition.
    We give necessary and sufficient conditions for $\cok(G) = 2$, which is the minimum size of a $k$-coalition of $G$ when $|V(G)| \neq k$. We denote the set of possible sizes of the $k$-coalitions by $\kset(G)$, and show the following result. 
    \begin{thm:interval}
        The set $\kset(G)$ is the interval $[\cok(G),\COk(G)]$. 
    \end{thm:interval}

    \item[\textbf{\Cref{sect: k-coalition partition graphs}}] generalizes the notion of a coalition partition graph, studied in \cite{HHHMM23CCO}, \cite{HHHMM23DMGT}, and \cite{HHHMM23OM}. In \cite{HHHMM23CCO}, the authors demonstrate that every graph is a coalition graph (see \Cref{sect: k-coalition partition graphs} for the definition). We generalize the construction to prove the following theorem.
    \begin{thm:kco graph}
        For every graph $H$ and every $k \geq2$, there exists a graph $G$ and a $k$-coalition partition $\copart$ of $G$ such that the $k$-coalition graph of $G$ and $\copart$ equals $H$.
    \end{thm:kco graph}
    Furthermore, we structurally characterize all $k$-coalition graphs when $k =\delta(G)$  and $k >\delta(G)$, and show that every graph admits a $k$-coalition partition whose associated $k$-coalition graph is the complete graph.
\end{enumerate}

\section{Definitions and Notation} \label{sect: definitions and notation}

We let $G=(V(G), E(G))$ be a finite, simple, and undirected graph. We assume $G$ has at least two vertices, unless noted otherwise. The complement  of $G$, denoted by $\overline{G}$, is the graph on the same vertex set $V(G)$, where two vertices are adjacent in $\overline{G}$ if and only if they are not adjacent in $G$.
The \defn{open neighborhood} of a vertex $v$ is given by $ N_G(v) = \{ u \mid uv \in E(G) \}$ and the \defn{closed neighborhood} is $N_G[v] = N_G(v) \cup \{v\}$; we will omit the subscript when $G$ is evident from the context. We let $[a,b]$ denote the integer interval $\{n \in \ints \mid a \leq n \leq b\}$.
Two subsets $S, T \subseteq V(G)$ are \defn{intersecting} if $S \cap T \neq \emptyset$. If $|S| = 1$, then we call $S$ a \defn{singleton set}.
A set of vertices $S \subseteq V(G)$ is a \defn{dominating set} if every vertex in $V(G) \setminus S$ has a neighbor in $S$. 
The complete graph on $n$ vertices is denoted by $K_n$. The \defn{complete bipartite graph}, denoted by $K_{m,n}$, has vertex set $V(K_{m,n}) = S_1 \cup S_2$, where $|S_1| = m$, $|S_2| = n$, and two vertices are adjacent if and only if they belong to distinct partite sets. More generally, for $r\geq 2$, the \defn{complete $r$-partite graph} $K_{s_1, \dots, s_r}$ is the graph with vertex set $V = S_1 \cup S_2 \cup \dots \cup S_r$, where each $|S_i|= s_i$ and $2 \leq s_1 \leq s_2 \leq \dots \leq s_r$. Two vertices  $u \in S_i$ and $v \in S_j$ are adjacent if and only if $i \neq j$. The \defn{star} $S_n$ is the complete bipartite graph $K_{n,1}$.
A graph $G$ is \defn{$r$-regular} if every vertex in $V(G)$ has degree $r$.

\begin{definition}
    The \defn{join} of graphs $G$ and $H$, denoted $G + H$, is the graph with vertex set $V(G) \cup V(H)$ and edge set consisting of $E(G) \cup E(H)$ along with all possible edges joining a vertex of $G$ and a vertex of $H$. In other words, the edge set of $G+H$ is given by $$ E(G) \cup E(H) \cup \{gh\mid g\in V(G) \text{ and } h\in V(H)\}.$$ The \defn{disjoint union} of $G$ and $H$, denoted $G \sqcup H$, is the graph with vertex set $V(G) \cup V(H)$ and edge set $E(G) \cup E(H)$. 
    See \Cref{fig:Graph operation example} for examples.
\end{definition}

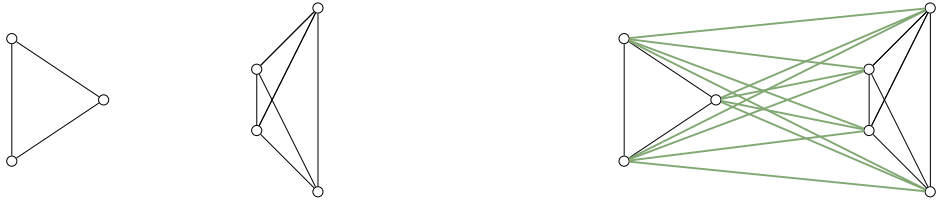
\begin{figure}[htb]
    \centering 
    \scalebox{0.9}{
       \begin{tikzpicture}[scale=1.5, every node/.style={circle, draw, fill=white, inner sep=1.5pt},  scale = 1.2
    ]

\node (A1) at (1,0.25) {};
\node (A2) at (1,1.25) {};
\node (A3) at (1.75,0.75) {};

\draw (A1) -- (A2) -- (A3) -- (A1);

\node (B1) at (3,0.5) {};
\node (B2) at (3,1) {};
\node (B3) at (3.5,1.5) {};
\node (B4) at (3.5,0) {};

\draw (B1) -- (B2) -- (B3) -- (B1);
\draw (B4) -- (B1) -- (B3);
\draw (B4) -- (B2) -- (B3);
\draw (B4) -- (B3);

\node (C1) at (6,0.25) {};
\node (C2) at (6,1.25) {};
\node (C3) at (6.75,0.75) {};

\draw (C1) -- (C2) -- (C3) -- (C1);

\node (D1) at (8,0.5) {};
\node (D2) at (8,1) {};
\node (D3) at (8.5,1.5) {};
\node (D4) at (8.5,0) {};

\draw (D1) -- (D2) -- (D3) -- (D1);
\draw (D4) -- (D1) -- (D3);
\draw (D4) -- (D2) -- (D3);
\draw (D4) -- (D3);

\foreach \i in {C1,C2,C3}
  \foreach \j in {D1,D2,D3,D4}
    \draw[draw = OliveGreen!60, thick] (\i) -- (\j);

\end{tikzpicture}}
    \caption{Left: The disjoint union $C_3 \sqcup K_4$. Right: The graph join $C_3 + K_4$.}
    \label{fig:Graph operation example}
\end{figure}

We will abuse notation and let $V(G)$ and $V(H)$ refer to the vertices of the copies of $G$ and $H$, respectively, in $G \sqcup H$ or $G + H$.

\section{Bounds on $k$-coalition number} \label{sect: Bounds on k-coalition number}
Previous work on $k$-coalitions by Jafari, Alikhani, and Bakhshesh \cite{JAB25} has primarily focused on establishing bounds for the $k$-coalition number of various graph classes in terms of minimum and maximum degree. We summarize some of their key results and tools below (and in the following section), as they will be useful.

We begin with an easy observation. 
\begin{observation} \label{Observation: No max size partition has k-dom set}
    Let $k \geq 2$. Suppose $G$ is a graph and $\copart$ is a $k$-coalition partition with maximum cardinality $\COk(G)$. Then no set in $\copart$ is a $k$-dominating set with exactly $k$ vertices. Equivalently, every set in $\copart$ forms a $k$-coalition with some other set.
\end{observation}

\begin{proof}
    Let $\copart$ be some maximum-cardinality $k$-coalition partition of $G$. Suppose, for contradiction, that $X \in \copart$ is a $k$-dominating set of $G$ with $k$ vertices. Then there is no set that forms a $k$-coalition with only $X$. Since $X$ has $k\geq 2$ vertices, we may split $X$ into two nonempty sets $X_1$ and $X_2$. Then $X_{1}$ and $X_{2}$ form a $k$-coalition.  Let $\copart' = \copart \setminus\{X\} \cup \{X_1, X_2\}$. Then $\copart'$ is a $k$-coalition partition of $G$ with cardinality $\COk(G) + 1$, contradicting the maximality of $\copart$.
\end{proof}

The authors compute the $k$-coalition number for complete graphs $K_n$; we restate the result below. 
\begin{lemma} \cite[Observation~3.1]{JAB25} \label{Thm: CO_k(K_n)}
    For every $2 \leq k \leq n-1$, the $k$-coalition number $\COk(K_n) = n-k+2$.
\end{lemma}

When $k > \Delta(G)$, no vertex has k neighbors, so a k-dominating set must contain
every vertex, so we obtain the following observation.
\begin{observation} \label{Observation: when k > max degree}
    Let $G$ be a graph. If $k > \Delta(G)$, then $\COk(G) = 2$.
\end{observation} 

We now consider $k \leq \Delta(G)$. 
\begin{lemma} [{\cite[Lemma~2.4] {JAB25}}] \label{Lemma: max number of coalition partners}
Let $\copart$ be a $k$-coalition partition of a graph $G$. If $k \leq \Delta(G)$, then each set $X \in \copart$ forms a $k$-coalition with at most $\Delta(G) -k +2$ sets of $\copart$.   
\end{lemma} 

When $k>\delta(G)$, \Cref{Lemma: max number of coalition partners} directly implies the following upper bound on the $k$-coalition number, since every set in the $k$-coalition partition must form a $k$-coalition with the set containing the vertices of degree $\delta(G)$.

\begin{lemma}[{\cite[Theorem~2.5] {JAB25}}] \label{Theorem: k > min deg CO_k lower bound}
    For any graph $G$ with maximum degree $\Delta(G)$, minimum degree $\delta(G)$, and $k > \delta(G)$, we have $$\COk(G) \leq \Delta(G) - k +3. $$
\end{lemma} 

When $k \leq \delta(G)$, Bre\v{s}ar, Henning, and Samadi \cite{BHS25} obtain the following lower bound on the $k$-coalition number.

\begin{lemma} [{\cite[Theorem~2.1]{BHS25}}]\label{Theorem: min degree lower bound on CO_k}
    For any graph $G$ with minimum degree $\delta(G)$ and $ 2 \leq k \leq \delta(G)$, $$\COk(G) \geq \delta(G) - k +3.$$ 
\end{lemma}

It follows from \Cref{Theorem: min degree lower bound on CO_k} that sufficiently large minimum degree guarantees large $k$-coalition number. 
Therefore, there does not exist a graph class with unbounded minimum degree and bounded $k$-coalition number. However, there exist graph classes with unbounded maximum degree and bounded $k$-coalition number. For example, for the class of star graphs $\{S_n\mid n \geq 1\}$ we have $\COk(S_n) =2$ for all $n$.

\begin{theorem} [{\cite[Theorem~2.8]{JAB25}}] \label{Theorem: Regular Coalition Bound}
Let $G=(V,E)$ be a graph with $|V| \geq 3$ and $k \geq 1$. If $G$ is a $k$-regular graph, then $3 \leq \COk(G) \leq 4$. 
\end{theorem}

As an application of the results above, we bound and compute the $k$-coalition number of two graph classes: trees and hypercubes.

\begin{theorem} Let $k \geq 2$. 
\begin{enumerate}
    \item For every tree $T$ with maximum degree $\Delta$ and $k \leq \Delta$, we have $\COk(T) \leq \Delta - k +3$.
   \item For every $k \leq \Delta$ and $n \in [2, \Delta - k +3]$, there exists a tree $T$ such that $\COk(T) = n$ and $\Delta(T) = \Delta$. In particular, there exists a tree such that equality holds in (1).
   \end{enumerate}
\end{theorem}

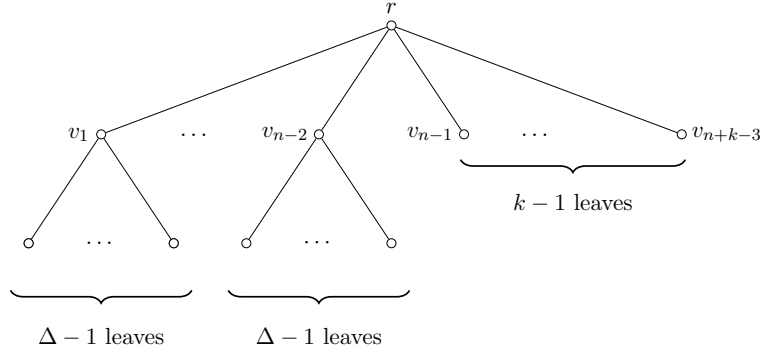
\begin{figure} [htb]
    \centering
    \scalebox{0.8}{
    \begin{tikzpicture}[scale=1.5, every node/.style={circle, draw, fill=white, inner sep=1.5pt},  scale = 0.8
    ]

\node (r) at (0,0) [label=above:$r$] {};

\node (v1) at (-4,-1.5) [label=left:$v_1$] {};
\node (vn2) at (-1,-1.5) [label=left:$v_{n-2}$] {};
\node (vn1) at (1,-1.5) [label=left:$v_{n-1}$] {};
\node (vnk) at (4,-1.5) [label=right:$v_{n+k-3}$] {};

\node[draw=none, fill=none] at (-2.7,-1.5)  {$\dots$};
\node[draw=none, fill=none] at (2,-1.5)   {$\dots$};

\draw (r) -- (v1);
\draw (r) -- (vn2);
\draw (r) -- (vn1);
\draw (r) -- (vnk);

\node (u11) at (-5,-3) {};
\node (u1d) at (-3,-3) {};
\node[draw=none, fill=none] at (-4,-3) {$\dots$};

\draw (v1) -- (u11);
\draw (v1) -- (u1d);

\node (un21) at (-2,-3) {};
\node (un2d) at (0,-3) {};
\node[draw=none, fill=none] at (-1,-3) {$\cdots$};

\draw (vn2) -- (un21);
\draw (vn2) -- (un2d);

\node [draw = none] (bvn1) at (1,-1.8)  {};
\node [draw = none]  (bvnk) at (4,-1.8) {};
\draw [decorate, decoration={brace, mirror, amplitude=6pt}, thick]
  (bvn1.south west) -- (bvnk.south east) node[draw = none, fill = none, midway, below=-10pt] {$k-1$ leaves};

  \node (u11) at (-5,-3) {};
\node (u1d) at (-3,-3) {};

\node [draw = none] (bun11) at (-5.2,-3.6)  {};
\node [draw = none]  (bun1d) at (-2.8,-3.6) {};
\draw [decorate, decoration={brace, mirror, amplitude=6pt}, thick]
  (bun11.south west) -- (bun1d.south east) node[draw = none, fill = none, midway, below=-10pt] {$\Delta-1$ leaves};

\node [draw = none] (bun21) at (-2.2,-3.6)  {};
\node [draw = none]  (bun2d) at (0.2,-3.6) {};
\draw [decorate, decoration={brace, mirror, amplitude=6pt}, thick]
  (bun21.south west) -- (bun2d.south east) node[draw = none, fill = none, midway, below=-10pt] {$\Delta-1$ leaves};
\end{tikzpicture}
    }
    \caption{Construction of the tree $T_{\Delta, k, n}$, which has maximum degree $\Delta$ and $k$-coalition number equal to $n$.}
    \label{fig:Tree Tight}
\end{figure}

\begin{proof}
Fix $k \leq \Delta$ and let $T$ be a tree. Since $\delta(T) = 1$, \Cref{Theorem: k > min deg CO_k lower bound} implies that $\COk(T) \leq \Delta - k +3$. 
We now construct, for each $n \in [2, \Delta - k +3]$, a tree $T_{\Delta, k,n}$ with $\Delta(T_{\Delta, k,n})= \Delta$ and $\COk(T_{\Delta, k,n}) = n$.

If $n = 2$, then let $T_{\Delta, k,n}$ be the star $S_{\Delta}$.
For $n \geq 3$, we construct $T_{\Delta, k,n}$ as follows.
Let $r$ be the root and $\{v_1, \dots, v_{n + k -3}\}$ be the neighbors of $r$. For each $i \in \{1, \dots, n-2\}$, attach $\Delta-1$ leaves to $v_i$, and let $\{v_{n-1}, \dots, v_{n+k-3}\}$ remain as leaves. See \Cref{fig:Tree Tight} for the construction.
The vertices $v_1, \dots, v_{n-2}$ have degree $\Delta$, and at least one such vertex exists since $n \geq 3$. Moreover,
$$\deg(r) = n + k -3 \leq (\Delta -k +3) + k -3 \leq \Delta.$$
Thus, $\Delta(T_{\Delta, k, n}) = \Delta$.
Let $L$ be the set of leaves and 
\begin{equation*}
    \copart =  \left\{  L, \{v_1\}, \dots, \{v_{n-2}\}, \{r\} \right\}.
\end{equation*}
When $k \leq  \Delta - 1$, every non-root vertex is $k$-dominated by $L$ and $r$ is $k-1$ dominated by $L$, so each singleton set forms a $k$-coalition with $L$. If $k=\Delta$, then $n$ is at most $3$, and $\copart = \{ \{r\}, \{v_1\}, L\}$. In either case, $\copart$ is a $k$-coalition partition. 

Suppose, for contradiction, there is a $k$-coalition partition $\copart'$ with more than $n$ sets. If $\copart'$ has a set containing $L$, then, even if every other set is a singleton set, $\copart'$ has at most $n$ sets.
Therefore, $\copart'$ does not have a set containing $L$. Since $k \geq 2$, every $k$-dominating set of $T_{\Delta, k,n}$ must contain $L$. Since every $k$-coalition pair must jointly contain $L$, we conclude $|\copart'| = 2 \leq n$. 
When $n = \Delta - k  +3$, the tree $T_{\Delta, k,n}$ satisfies $\COk(T_{\Delta, k,n}) = \Delta - k + 3$. Thus, the bound in \Cref{Theorem: k > min deg CO_k lower bound} is tight for all values of $\Delta(G)$ and $k\geq 2$.
\end{proof}

The $n$-dimensional hypercube $Q_n$ has vertex set $V(Q_n)=  \{(e_1, \dots, e_n)\mid e_i \in \ints_2 \text{ for all } i \}$. Two vertices $(e_1, \dots, e_n)$ and $(e_1', \dots, e_n')$ in $V(Q_n)$ are adjacent if and only if they differ in exactly one coordinate. For example, $Q_2 = C_4$.

\begin{theorem}
    Let $n\geq 2$ and let $Q_n$ be the $n$-dimensional hypercube. Then $CO_n(Q_n) = 4$.
\end{theorem}

\begin{proof}
    Since hypercube $Q_n$ is $n$-regular, it follows from \Cref{Theorem: Regular Coalition Bound} that $3 \leq CO_n(Q_n)\leq 4$. Proceeding by induction, for each $n$ we will exhibit an $n$-coalition partition with $4$ sets. Observe that the following is a $2$-coalition partition of $Q_2$: $$\left\{ \{(1,1)\},  \{(1,0)\},  \{(0,1)\},  \{(0,0)\} \right\},$$ where $ \{(1,1)\}$ and $ \{(0,0)\}$ form one $2$-coalition and $ \{(1,0)\}$ and $ \{(0,1)\}$ form another. See \Cref{fig: 3D Hypercube Example} for a $3$-coalition partition of $Q_3$.
    
   Suppose that we have an $n$-coalition partition $\copart_n = \{X_1, X_2, X_3, X_4\}$ of $Q_n$, where $X_1$ and $X_2$ form one $n$-coalition pair and $X_3$ and $X_4$ form another, such that $|X_i|= 2^{n-2}$ for all $1 \leq i \leq 4$. Then we inductively construct an $(n+1)$-coalition partition of $Q_{n+1}$. Set
    \begin{align*}
        Y_1 = \{ (e_1, \dots, e_n, 0) \in V(Q_{n+1}) \mid (e_1, \dots, e_n) \in X_1 \cup X_2\}; \\
        Y_2 = \{ (e_1, \dots, e_n, 1) \in V(Q_{n+1}) \mid (e_1, \dots, e_n) \in X_3 \cup X_4\}; \\
        Y_3 = \{ (e_1, \dots, e_n, 0) \in V(Q_{n+1}) \mid (e_1, \dots, e_n) \in X_3 \cup X_4\}; \\
        Y_4 = \{ (e_1, \dots, e_n, 1) \in V(Q_{n+1}) \mid (e_1, \dots, e_n) \in X_1 \cup X_2\}.
    \end{align*}
 We check that $\copart_{n+1} = \{Y_1, Y_2, Y_3, Y_4\}$ forms an $(n+1)$-coalition partition of $Q_{n+1}$ that satisfies the inductive hypothesis. Observe that $Y_1$ and $Y_2$ each $n$-dominate the subsets $(\ints_2)^n \times \{0\}$ and $(\ints_2)^n \times \{1\}$, respectively, and yet neither is an $(n+1)$-dominating set of $Q_{n+1}$. Let $v= (e_1, \dots, e_n, e_{n+1}) \in V(Q_{n+1}) \setminus (Y_1 \cup Y_2)$. Suppose $e_{n+1} = 0$. Then $(e_1, \dots, e_n, e_{n+1}) \in X_3 \cup X_4$. Hence, $v$ has $n$ neighbors in $Y_1$ and one neighbor in $Y_2$, so $v$ is $(n+1)$-dominated by $Y_1 \cup Y_2$. Analogous reasoning holds when $e_{n+1}=1$. Therefore, $Y_1$ and $Y_2$ form an $(n+1)$-coalition, and similarly $Y_3$ and $Y_4$ form an $(n+1)$-coalition. Thus, $\{Y_1, Y_2, Y_3, Y_4\}$ is an $(n+1)$-coalition partition, as required.
\end{proof}

\begin{figure} [htb]
    \centering
    \begin{tikzpicture}[
  every node/.style={circle, draw, fill=white, inner sep=1.5pt}, 
  scale=1.2
]

\node (A1) at (0, 0) {};
\node (A2) at (0,1) {};
\node (A3) at (1,1) {};
\node (A4) at (1,0) {};

\draw (A1) -- (A2) -- (A3) -- (A4) -- (A1);


\node (B1) at (-0.5, -0.5) {};
\node (B2) at (-0.5,1.5) {};
\node (B3) at (1.5,1.5) {};
\node (B4) at (1.5,-0.5) {};

\draw (A1) -- (B1) ;
\draw (A2) -- (B2) ;
\draw (A3) -- (B3) ;
\draw (A4) -- (B4) ;

\draw (B1) -- (B2) -- (B3) -- (B4) -- (B1);

\begin{scope} [on background layer]
    \node[ellipse, draw=Cerulean, thick, fill=  Cerulean!30, minimum width=0.6cm, minimum height=0.6cm, fit=(A1)] {};
    \node[ellipse, draw=Cerulean, thick, fill=  Cerulean!30, minimum width=0.6cm, minimum height=0.6cm, fit=(A3)] {};
    \node[ellipse, draw=OliveGreen,  fill=LimeGreen!30, minimum width=0.6cm, minimum height=0.6cm, fit=(A2)] {};
    \node[ellipse, draw=OliveGreen,  fill=LimeGreen!30, minimum width=0.6cm, minimum height=0.6cm, fit=(A4)] {};
      \node[ellipse, draw=CarnationPink, fill=CarnationPink!30, minimum width=0.6cm, minimum height=0.6cm, fit=(B1)] {};
    \node[ellipse, draw=CarnationPink, fill=CarnationPink!30, minimum width=0.6cm, minimum height=0.6cm, fit=(B3)] {};
    
    \node[ellipse, draw=Goldenrod, fill=Goldenrod!30, minimum width=0.6cm, minimum height=0.6cm, fit=(B2)] {};
    \node[ellipse, draw=Goldenrod, fill=Goldenrod!30, minimum width=0.6cm, minimum height=0.6cm, fit=(B4)] {};
\end{scope}

\end{tikzpicture}
\caption{A $3$-coalition partition of $Q_3$ with four sets: blue, green, pink, and yellow.}
    \label{fig: 3D Hypercube Example}
\end{figure}
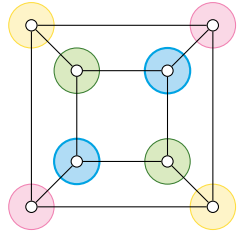

Jafari, Alikhani, and Bakhshesh in \cite{JAB25} provide lower and upper bounds for the $k$-coalition number of complete bipartite graphs.  Bre\v{s}ar, Klav\v{z}ar, and Samadi \cite{BHS25} computed the $k$-coalition number exactly for complete bipartite graphs.

\begin{theorem}[{\cite[Theorem~5.2] {BHS25}}] \label{Theorem: Complete bipartite graphs} Let $s_2 \geq s_1$ and $k \geq 2$. Then,
    \begin{align*}
\COk(K_{s_1,s_2}) =\begin{cases}
    s_2-k+3 & \text{if $s_1 \leq 3k-2 $}, \\
    s_1 +s_2 - 4k + 4 & \text{if $s_1 \geq 3k-1 $}.
\end{cases}
\end{align*}
\end{theorem}

We extend this work by computing the $k$-coalition number for complete multipartite graphs with at least 3 partite sets. See \Cref{sect: complete multipartite graphs} for the proof.

\begin{theorem} \label{Theorem: Multipartite k-coalition number}
Fix $k\geq 2$. Let $K_{s_1,s_2, \dots, s_r}$ be a complete $r$-partite graph where $2\leq s_1\leq s_2 \leq \dots \leq s_r$. If $k < s_1$, then 
\begin{align*}
\COk(K_{s_1,\dots, s_r}) =\begin{cases}
    \sum_{i=1}^{r}s_i - 2k + 2 & \text{if $r \geq 2k $}, \\
    \sum_{i=1}^{r}s_i - 2k + 1 & \text{if $k < r < 2k$ and $s_1 \geq k + 2$}, \\
    \sum_{i=1}^{r}s_i - 2k - 2\big\lceil \frac{k}{r-1}\big\rceil+ 4 & \text{if $r \leq k$ and $\big(k + \big\lceil \tfrac{k}{r-1}\big\rceil -1\big) \bmod{r} \leq \tfrac{r}{2}$}  \\ 
    &\text{and $s_1 \geq k + 2\big\lceil \tfrac{k}{r-1}\big\rceil-1$}, \\
    \sum_{i=1}^{r}s_i - 2k - 2\big\lceil \frac{k}{r-1}\big\rceil+ 3 & \text{if $r \leq k$ and $\big(k + \big\lceil \frac{k}{r-1}\big\rceil -1 \big) \bmod{r} > \frac{r}{2}$} \\ &\text{and $s_1 \geq k + 2\big\lceil \frac{k}{r-1}\big\rceil$},  \\
    \sum_{i=2}^{r}s_i - k + 3 & \text{otherwise}. 
\end{cases}
\end{align*}
\end{theorem}

\section{Graph Operations} \label{sect: graph operations}
Jafari, Alikhani, and Bakhshesh \cite{JAB25} propose investigating the behavior of the $k$-coalition number under graph operations. In this section, we show that the $k$-coalition number is superadditive under the graph join and (strictly) subadditive under the disjoint union. When the $k$-coalition partitions of $G$ and $H$ are structured ``nicely,'' we can compute exactly the $k$-coalition number of $G \sqcup H$ in terms of $\COk(G)$ and $\COk(H)$.

We first show that the $k$-coalition number is superadditive under the graph join. 
\begin{proposition} \label{Prop: Join superadditive}
    Let $k\geq 1$ and let $G$ and $H$ be graphs each having at least $k$ vertices. Then $$\COk(G + H) \geq \COk(G) + \COk(H).$$ 
\end{proposition}

\begin{proof}
    Fix maximum–cardinality $k$-coalition partitions of $G$ and $H$, say $\copart_G = \{X_1, \dots, X_{\COk(G)}\}$ and $\copart_H = \{ Y_1, \dots, Y_{\COk(H)}\}$, respectively. We will show that $\copart_G \cup \copart_H$ is a $k$-coalition partition of $G + H$. By \Cref{Observation: No max size partition has k-dom set}, there are no $k$-dominating sets in  $\copart_G$ or $\copart_H$.
    Suppose $X_1, X_2 \in \copart_G$ are coalition partners. Since neither $X_1$ nor $X_2$ can $k$-dominate $V(G)$, neither set $k$-dominates $V(G + H)$. Since the union $X_1 \cup X_2$ $k$-dominates $V(G)$, it follows that $|X_1 \cup X_2| \geq k$. Since every vertex in $V(H)$ is adjacent to every vertex in $X_1 \cup X_2$, the union $k$-dominates $V(H)$ and therefore forms a $k$-coalition in $G + H$. By analogous reasoning, every $k$-coalition pair in $\copart_H$ is also a $k$-coalition pair in $G + H$. Thus, $\copart_{G} \cup \copart_{H}$ is a $k$-coalition partition of $G+H$ and $\COk(G + H) \geq |\copart_{G} \cup \copart_{H}| = \COk(G) + \COk(H)$.
\end{proof}

In fact, when we repeatedly take the join of a graph $G$ with itself, the $k$-coalition number increases linearly with the number of vertices in $G$. We begin by proving a more general statement.

\begin{theorem} \label{Theorem: Join Limit} 
Fix $2 \leq k < n$.
   Given $m \geq 2$ (possibly distinct) $n$-vertex graphs $G_1, G_2, \dots, G_m$, we have $$\COk(G_1+ \dots+ G_m) \geq (m-1)n - k + 2.$$
\end{theorem}

\begin{proof}
    First suppose that $G_1,\dots, G_m$ are all $n$-vertex cliques. Since the join of cliques is a clique, we have $G_1+ \dots+ G_m = K_{mn}$. By \Cref{Thm: CO_k(K_n)}, $\COk(G_1+ \dots+ G_m) = mn -k +2 \geq (m-1)n - k + 2$. 
    
    Now suppose that some $G_i$ is not a clique, which we may relabel to be $G_1$. Then there exist two vertices $v_1, v_2 \in V(G_1)$ that are nonadjacent. Pick a subset $X_1 \subseteq V(G_1)$ with $|X_1| = k$, $v_1\in X_1$, and $v_2 \notin X_1$. Then $X_1$ is not a $k$-dominating set of $G_1 + \dots + G_m$ (since $X_1$ does not $k$-dominate $v_2$), but $X_1$ does $k$-dominate $V(G_1 + \dots + G_m) \setminus V(G_1)$. Pick $X_2 \subseteq V(G_2)$ such that 
    $X_1 \cup X_2$ is a $(k-1)$-dominating set but not a $k$-dominating set of $G_1 + \dots + G_m$ (it may be that $X_2 = \emptyset$). 
    Note that $|X_2| \leq k -1$. 
    For each $u \in V(G_1) \setminus X_1$, pick a distinct $u' \in V(G_2) \setminus X_2$ (which we may do since $|V(G_2) \setminus X_2| \geq |V(G_1) \setminus X_1|$), and set $Y_u = \{u, u'\}$.  Note that $ \{u, u'\}$ forms a $k$-coalition with $X_1 \cup X_2$. 
  Every singleton $\{ v\} \nsubseteq X_1 \cup X_2 \cup \{ Y_u \mid u \in V(G_1) \setminus X_1\}$, forms a $k$-coalition with $X_1 \cup X_2$. Therefore, the following is a $k$-coalition partition of $G_1 + \dots + G_m$
    \begin{align*}
        \copart = \left\{ X_1 \cup X_2\right\} \cup \left\{ Y_u \mid u \in V(G_1) \setminus X_1 \right\}  \cup \left\{ \{v\} \mid v \notin ( X_1 \cup X_2 \cup \{ Y_u \mid u \in V(G_1) \setminus X_1 \}) \right\}
    \end{align*}
 and so
    \begin{equation*}
        \COk(G_1 + \dots + G_m) \geq |\copart| = |V(G_1+ \dots+ G_m)| - |V(G_1)| - |X_2|  +1 \geq  (m-1)n - k + 2. \qedhere
    \end{equation*}
\end{proof}

\begin{corollary} 
    Fix $k,n,m \in \N$ such that $2\leq k<n$ and $m \geq 2$. For any $n$-vertex graph $G$,
    $$\COk(mG) \geq (m-1)n-k+2,$$
    where $mG = G+ \dots + G$ denotes the join of $G$ with itself $m$ times. 
\end{corollary}

It turns out to be much harder to prove that the $k$-coalition number is strictly subadditive under the disjoint union. To obtain an upper bound, we start with an arbitrary maximum-size $k$-coalition partition and then carefully move vertices between sets without decreasing the size of the partition. We require the following lemma. 



\begin{lemma} \label{Lemma: Converting dominating sets to coalitions}
    Fix $k \geq 2$. Given a non-empty collection $\mathcal{D} = \{D_1,\dots, D_n\}$ of disjoint $k$-dominating sets of a graph $G$, there is a collection of non-empty disjoint sets $\pi = \{X_1,\dots, X_m\}$ such that every $X_i\in \pi$ forms a $k$-coalition with some $X_j \in \pi$, $\bigcup_{i=1}^nD_i = \bigcup_{j=1}^m X_j$, and $|\pi| >|\mathcal{D}|$.
\end{lemma}

\begin{proof}
We may assume that $D_1, \dots, D_{n-1}$ are minimal $k$-dominating sets of $V(G)$. If not, then there exists some minimal $k$-dominating set $D_i' \subset D_i$. We replace $D_i$ with $D_i'$ and add all vertices of $D_i\setminus D_i'$ to $D_{n}$. Since $|D_i| \geq k \geq 2$, we may partition each minimal $k$-dominating set $D_i$ into two $k$-coalition partners $X_{i, 1}$ and $X_{i, 2}$. 
If $D_n$ is not minimal, then there exists a minimal $k$-dominating set $D_n^1 \subseteq D_n$, split it into $X_{n,1}^1$ and $X_{n,2}^1$, and replace $D_n$ by $D_n' = D_n\setminus D_n^1$. Repeat this step until $D_n'$ is no longer a $k$-dominating set. Suppose that this requires $N$ iterations. Let $$\pi = \left( \bigcup _{i=1}^{n} \{X_{i,1}, X_{i,2}\} \right) \cup \left( \bigcup_{j=1}^{N} \{X_{n,1}^j,X_{n,2}^j \}\right).$$ 
If the remainder of $D_n'$ does not form a coalition with any set in $\pi$, then replace $X_{n,1}^1$ with $X_{n,1}^1 \cup D_n'$. The union is not a $k$-dominating set, since this would indicate that $X_{n,1}^1$ and $D_n'$ would have been $k$-coalition partners. If $D_n'$ forms a coalition with some set in $\pi$, we add $D_n'$ to $\pi$.  Since every $k$-dominating set is divided into least $2$ sets in this process, we deduce $|\pi| > |\mathcal{D}|$. We may re-enumerate the sets to obtain that $\pi = \{X_1, \dots, X_m\}$.
\end{proof}

\begin{theorem} \label{Theorem: Disjoint Union subadditive}
    Let $G$ and $H$ be graphs with more than $k$ vertices and let $k \geq 2$. Then $$\COk(G \sqcup H) \leq \COk(G) + \COk(H) - 2.$$
\end{theorem}
\begin{proof}
    Let $\copart = \{X_1, \dots, X_{\COk(G\sqcup H)}\}$ be a maximum-size $k$-coalition partition of $G\sqcup H$. For simplicity, we will write $S \cap V(G) = \{ X \cap V(G) \mid X\in S \text{ and } X \cap V(G) \neq \emptyset\}$
    for any subset $S \subseteq \copart$. By \Cref{Observation: No max size partition has k-dom set}, no set individually $k$-dominates both $V(G)$ and $V(H)$. 

\sta{\label{Claim: No k-dom sets of just V(G)} 
   If no set $X \in \copart$ individually $k$-dominates $V(G)$, then $|\copart| < \COk(H) + \COk(G)$. The same holds for $V(H)$.}

Note that every set $X \in \copart$ intersects $V(G)$. Suppose otherwise that there is some set $X$ such that $X \cap V(G) = \emptyset$. Then $X$ must have a $k$-coalition partner $X'$ such that $X \cup X'$ $k$-dominates $G \sqcup H$. Therefore, $X'$ $k$-dominates $V(G)$ by itself, a contradiction. 
    
 We now show that restricting $\copart$ to $V(G)$ yields a valid $k$-coalition partition of $V(G)$.
Let $\copart_G = \copart \cap V(G)$. Every set $X \in \copart$ has a $k$-coalition partner $X'$, where $X \cup X'$ is a $k$-dominating set of $V(G)$, yet neither $X$ nor $X'$ $k$-dominate $V(G)$ by assumption. Therefore, $\copart_G$ is a $k$-coalition partition of $G$ with exactly $|\copart|$ sets and $|\copart| = |\copart_G| \leq \COk(G)$. Since $\COk(H)\geq 1$, we conclude that $|\copart| < \COk(H) + \COk(G)$. Analogously, if no set in $\copart$ individually $k$-dominates $V(H)$, then $|\copart| \leq \COk(H) < \COk(H) + \COk(G)$. This proves \eqref{Claim: No k-dom sets of just V(G)}.

\sta{\label{Claim: Case with k-dom set of V(G) and V(H))} 
   If $\copart$ contains a set $X_1$ that $k$-dominates $V(G)$ and a set $X_2$ that $k$-dominates $V(H)$, then $|\copart| < \COk(H) + \COk(G)$. }

 By \Cref{Observation: No max size partition has k-dom set}, there is no $k$-dominating set of $G \sqcup H$ in $\copart$, so $X_1 \neq X_2$.
    Define 
    \begin{align*}
    D_G &= \{X \in \copart \mid X \text{ $k$-dominates } V(G)\}, \\
    D_H & = \{X \in \copart \mid X \text{ $k$-dominates } V(H)\}, \\ 
    P_G &= \{X \in \copart \setminus (D_G \cup D_H) \mid X \text{ forms a $k$-coalition with some } X' \in D_H\},\\
    P_H & = \{X \in \copart \setminus (D_G \cup D_H) \mid X \text{ forms a $k$-coalition with some } X' \in D_G\},\\
    F & = \{X \in \copart \setminus (D_G \cup D_H) \mid X \text{ does not form a $k$-coalition with any } X' \in D_H \cup D_G \}.
    \end{align*}

    Note that $\copart = D_G \cup D_H \cup P_G \cup P_H \cup F$, and these subsets are pairwise disjoint with the exception of $P_G$ and $P_H$, as there may exist sets that form $k$-coalitions with sets in both $D_G$ and $D_H$. 

We first consider the restriction of $\copart$ to $V(G)$.
  Observe that $\copart_G$ is not necessarily a $k$-coalition partition of $V(G)$ because some sets in $\copart$ (that do not $k$-dominate $V(G)$) may only form $k$-coalitions with sets that do $k$-dominate $V(G)$. 
  With this in mind, we define
    \begin{align*}
        E_H &=  \{ X \in D_H \mid \text{ every $k$-coalition partner of $X$ is in } D_G \}, \\
        Q_H &= \{ X \in P_H \mid \text{ every $k$-coalition partner of $X$ is in }D_G\}. 
    \end{align*}
  
   If $X \in E_H \cup Q_H$, then $X$ does not $k$-dominate $V(G)$ but $X$ has no $k$-coalition partners in $\copart_G$. Since no set in $P_G$ or $F$ $k$-dominates $V(G)$ and each has a $k$-coalition partner that is not in $D_G$, these sets have a $k$-coalition partner in $\copart_G$.
   Therefore, the sets in $E_H$ and $Q_H$ are precisely the obstructions to $\copart_G$ being a $k$-coalition partition.
To address this, our strategy will be as follows: we will merge the vertices of $E_H$ and $Q_H$ into a single $k$-dominating set, then use \Cref{Lemma: Converting dominating sets to coalitions} to convert it to a collection of sets, where each has a $k$-coalition partner in the collection.

By assumption, there exists at least one set $X_1 \in D_G$. Let
\begin{equation*}
    X_1' = \big(X_1 \cap V(G)\big) \cup \bigg( \bigcup_{X\in E_H} X \cap V(G)\bigg) \cup \bigg(\bigcup_{X\in Q_H} X \cap V(G)  \bigg) .
\end{equation*}
Set $D_G' = (D_G \setminus \{X_1\}) \cup \{X_1'\}$ and note that $|D'_G| = |D_G|$.
Define $\overline{D_H} = D_H \setminus E_H$ and $\overline {P_H} = P_H \backslash Q_H$.
Then define $$\pi= \Big( D_G' \cup \overline{D_H} \cup  \big( P_G \setminus (P_G \cap P_H)  \big) \cup F \cup \overline{P_H} \Big) \cap V(G).$$
Note that we constructed $\pi$ by removing the sets that do not admit a $k$-coalition partner in $V(G)$, then reallocating the vertices from these sets to the $k$-dominating set $X_1$, which we now call $X_1'$. 
Every set $X \in \pi$ that is not a $k$-dominating set of $V(G)$ has at least one $k$-coalition partner $X' \in \pi$. The $k$-dominating sets in $\pi$ are precisely those in $D_G' \cap V(G)$. By \Cref{Lemma: Converting dominating sets to coalitions}, there is a collection of disjoint sets $D_G''= \{Y_1, \dots, Y_d\}$, such that every $Y_i$ forms a $k$-coalition with some $Y_j$, $\bigcup_{X\in D'_G} X = \bigcup_{i=1}^d Y_i$, and $|D_G''| > |D_G'|$.
Therefore, $\copart_G = [D_G'' \cup \overline{D_H} \cup P_G \cup \overline{P_H} \cup F ] \cap V(G)$ is a $k$-coalition partition of $V(G)$ and $|\copart_G| \leq \COk(G)$. By analogous reasoning, we get that $\copart_H = [D_H'' \cup \overline{D_G} \cup P_H \cup \overline{P_G} \cup F ] \cap V(H)$ is a $k$-coalition partition of $V(H)$ and $|\copart_H| \leq \COk(H)$. Combining these, we find that
\begin{align*}
    \COk(G \sqcup H) &= |\copart| \leq |D_G| + |D_H| + |P_G| + |P_H| + |F| \\
    & \leq \left(|D_G''| -1 + |\overline{D_H}| + |P_G| + |\overline{P_H}| +|F| \right)+  \left( |D_H''| -1 + |\overline{D_G}| + |P_H| +|\overline{P_G}| + |F|  \right) \\
    &= |\copart_G|  + |\copart_H| -2 \leq \COk(G) + \COk(H) -2,
\end{align*}
where the second inequality follows from the fact that $|D_G''| \geq |D_G'| + 1$ and $|D_H''| \geq |D_H'| +1$. The last equality follows by observing that the collections of sets that compose $\copart_G$ (and $\copart_H$) are disjoint. This proves \eqref{Claim: Case with k-dom set of V(G) and V(H))}.
 \end{proof}

We have proven an upper bound on the $k$-coalition number of the disjoint union of graphs. When the $k$-coalition partitions of a graph admit a certain ``nice'' structure, we can also lower-bound the $k$-coalition number of the disjoint union.

 \begin{proposition} \label{Prop: For nice partition of G CO_k(G disjoint H) geq CO_k(G)}
Let $G$ and $H$ be graphs and let $k \geq 2$. If $G$ admits a maximum-size $k$-coalition partition $\copart_G = \{ X_1, \dots, X_{\COk(G)}\}$, where $X_1$ forms a $k$-coalition with every $X_i$ for $i > 1$, then $$\COk(G \sqcup H) \geq \COk(G).$$
 \end{proposition}

\begin{proof} 
  Let $\copart_G = \{ X_1, \dots, X_{\COk(G)}\}$ be such a $k$-coalition partition. Then, by \Cref{Observation: No max size partition has k-dom set}, $\copart_G$ does not contain a $k$-dominating set of $G$. Set $$\copart_{G \sqcup H} = \{ X_1 \cup V(H)\} \cup \{X_2, \dots, X_{\COk(G)}\}.$$
   Observe that no set in $\copart_{G \sqcup H}$ individually $k$-dominates $G \sqcup H$, but every set forms a $k$-coalition with $X_1 \cup V(H)$. Then $\copart_{G \sqcup H}$ is a $k$-coalition partition of $V(G\sqcup H)$, and
   \begin{equation*}
       \COk(G\sqcup H) \geq |\copart_G| = \COk(G). \qedhere
   \end{equation*}
\end{proof}

By applying \Cref{Prop: For nice partition of G CO_k(G disjoint H) geq CO_k(G)} to both $G$ and $H$, we obtain the following result. 

\begin{corollary} \label{Cor: disjoint union lower bound by max of k-coalition numbers}
    Fix $k < \min\{|V(G)|, |V(H)|\}$. Suppose that $G$ and $H$ admit maximum-size $k$-coalition partitions $\copart_G = \{ X_1, \dots, X_{\COk(G)}\}$ and $\copart_H = \{ Y_1, \dots, Y_{\COk(H)}\}$, respectively, such that for each $i \in [1, \COk(G)]$ the set $X_i$ forms a $k$-coalition with $X_1$ and for each $j \in [1,\COk(H)]$ the set $Y_j$ forms a $k$-coalition with $Y_1$. Then $$\COk(G \sqcup H) \geq \max\{\COk(G), \COk(H)\}.$$
\end{corollary}

While the construction given in the proof of \Cref{Prop: For nice partition of G CO_k(G disjoint H) geq CO_k(G)} may appear wasteful, the following results demonstrate that the construction is optimal when $k$ is larger than the minimum degrees of $G$ and $H$.

\begin{proposition} \label{Prop: k > min degree G disjoint union H upper bound on CO_k}
     Fix $k < \min\{|V(G)|, |V(H)|\}$. If $k > \min\{ \delta(G), \delta(H)\}$, then $\COk(G \sqcup H) \leq \max\{\COk(G), \COk(H)\}$.
\end{proposition}
\begin{proof} \setcounter{tbox}{0}
 Let $V^- = \{ v \in V(G \sqcup H) \mid \deg(v) < k \}$. If $D$ is a $k$-dominating set of $G\sqcup H$, then $V^- \subseteq D$. 
Let $\copart = \{X_1, \dots, X_{\COk(G\sqcup H)}\}$ be a maximum-cardinality $k$-coalition partition of $G \sqcup H$, which has no $k$-dominating set by \Cref{Observation: No max size partition has k-dom set}.

\sta{\label{Claim: No set contains V-} If there is no set in $\copart$ that contains $V^-$, then $|\copart| \leq \max\{\COk(G) , \COk(H)\}$.}

The set $X_1$ must have a $k$-coalition partner, say $X_2$, where $X_1 \cup X_2$ is a $k$-dominating set. Then $V^- \subseteq X_1 \cup X_2$ and so every other $k$-dominating set must intersect $X_1$ and $X_2$. Since neither set individually contains $V^-$, it follows that $\copart = \{X_1, X_2\}$. Since $G$ and $H$ have more than $k$ vertices, each has $k$-coalition number at least $2$. Therefore, $|\copart| = 2 \leq \max\{\COk(G) , \COk(H)\}$, which proves \eqref{Claim: No set contains V-}.\\

We may now assume that there is a set, say $X_1$, that contains $V^-$. Then each $X_i \in \copart \setminus \{X_1\}$ can only form a $k$-coalition with $X_1$. By \Cref{Observation: No max size partition has k-dom set}, the set $X_1$ cannot $k$-dominate both $V(G)$ and $V(H)$. Without loss of generality, we assume that $X_1$ does not $k$-dominate $V(H)$. Let $\copart \cap V(H) = \{ X_i \cap V(H) \mid X_i\in \copart \text{ and } X_i \cap V(H) \neq \emptyset\}$.
For every $X_i \in \copart\setminus \{X_1\}$, the intersection $X_i \cap V(H)$ is nonempty, since the set $(X_i \cup X_1) \cap V(H)$ $k$-dominates $V(H)$. The set $X_i \cap V(H)$ is either a $k$-dominating set of $V(H)$ or forms a $k$-coalition with $X_1 \cap V(H)$. There are three cases:\\

\sta{\label{Case: Empty intersection} If $X_1 \cap V(H) = \emptyset$ (which may occur when $V^- \subseteq V(G)$), then $|\copart| \leq \COk(H)$.}
In this setting,
    \begin{equation}
        |\copart| =|\copart \cap V(H)| + 1. \label{eq: b}
    \end{equation}
        For every $X_i \in \copart$, the set $X_i \cap V(H)$ is a $k$-dominating set of $V(H)$. Hence, $\copart\cap V(H)$ is a set of $k$-dominating sets of $V(H)$. By \Cref{Lemma: Converting dominating sets to coalitions}, there is a collection of disjoint sets $\copart_H$ such that every set in $\copart_H$ has a $k$-coalition partner in $\copart_H$, $|\copart_H|> |\copart|$, and $\bigcup_{X \in \copart} X \cap V(H) = \bigcup_{Y \in \copart_H}Y$.
        Since $V(H) = \bigcup_{X \in \copart}X \cap V(H)$, $\copart_H$ is a $k$-coalition partition of $V(H)$. Then 
        \begin{equation} \label{eq: c}
           |\copart \cap V(H)| < |\copart_H| \leq \COk(H).
        \end{equation}
    Then \eqref{eq: b} and \eqref{eq: c} imply that $\COk(G \sqcup H) = |\copart| \leq \COk(H)$, which proves \eqref{Case: Empty intersection}.\\
       
\sta{\label{Case: Non-zero intersection 1}
If $X_1 \cap V(H) \neq \emptyset$ and some set $X_i \in \copart$ does not $k$-dominate $V(H)$, then $|\copart| \leq \COk(H)$.
}
In this case, $ |\copart| =|\copart \cap V(H)|$. If there are $k$-dominating sets of $V(H)$ in $\copart \cap V(H)$ (which must be minimal),  then by \Cref{Lemma: Converting dominating sets to coalitions}, we may split up the $k$-dominating sets into $k$-coalition partners. If we denote the resulting partition by $\copart_H$, \Cref{Lemma: Converting dominating sets to coalitions} implies that $|\copart| \leq |\copart_H|$. Since every set that is not a $k$-dominating set forms a $k$-coalition with $X_1$ and there is at least one such set, $\copart_H$ is a $k$-coalition partition. Therefore, $|\copart| =|\copart \cap V(H)| \leq |\copart_H| \leq \COk(H)$, which proves \eqref{Case: Non-zero intersection 1}.

\sta{\label{Case: Non-zero intersection 2}
If $X_1 \cap V(H) \neq \emptyset$ and, for every $i > 1$, the set $X_i \cap V(H)$ is a $k$-dominating set of $V(H)$, then $|\copart| \leq \COk(H)$.}

Note that $X_1$ does not have $k$-coalition partner in $\copart \cap V(H)$. Define 
        $$\copart' = (\copart \setminus \{X_1, X_2\}) \cup \{X_1 \cup X_2\}.$$
       Then $\copart' \cap V(H)$ is a set of $k$-dominating sets of $V(H)$. As above, by \Cref{Lemma: Converting dominating sets to coalitions} there is a collection of sets $\copart_H$ where every set in the collection has a $k$-coalition partner, $|\copart_H| \geq |\copart' \cap  V(H)|$, and $\copart_H$ is a partition of $V(H)$. Moreover, $\copart_H$ is a $k$-coalition partition  and so
        \begin{equation*}
            |\copart| = |\copart' \cap V(H)| +1 \leq |\copart_H | \leq \COk(H),
        \end{equation*}
        which proves \eqref{Case: Non-zero intersection 2}.\\

Therefore, we conclude that $\COk( G \sqcup H) = |\copart| \leq \COk(H)$. If $X_1$ does not $k$-dominate $V(G)$, by an analogous argument, $\COk(G \sqcup H) \leq \COk(G)$. Hence, 
\begin{equation*}
    \COk(G \sqcup H) \leq \max\{\COk(G), \COk(H)\}. \qedhere
\end{equation*}
\end{proof}

\begin{proposition}
      Fix $k < \min\{|V(G)|, |V(H)|\}$. If $k > \max\{\delta(G), \delta(H)\}$, then $$\COk(G \sqcup H) = \max\{\COk(G), \COk(H) \}.$$
\end{proposition}

\begin{proof}
 Let $V^-_G = \{ v \in V(G) \mid \deg(v) < k \}$ and $V^-_H = \{ v \in V(H) \mid \deg(v) < k \}$.  By assumption, $V^-_G\neq \emptyset$ and $V^-_H \neq \emptyset$. Set $V^- = V^-_G \cup V^-_H$. If $D$ is a $k$-dominating set of $G\sqcup H$, then $V^- \subseteq D$.\\
 
    \paragraph{\textbf{When $k$ is large}}  If $k >  \Delta(G \sqcup H) =\max\{\Delta (G), \Delta(H)\} $, then \Cref{Observation: when k > max degree} gives  $$\COk(G \sqcup H) = 2 = \max\{\COk(G), \COk(H)\}.$$
    We now assume that $k \leq \Delta(G \sqcup H)$.
\\
\paragraph{\textbf{Lower bound}} Let $\copart_G = \{X_1, \dots, X_{\COk(G)}\}$ be a maximum-cardinality $k$-coalition partition of $V(G)$. Suppose there exists a set $X_1 \in \copart_G$ such that $V^-_G \subseteq X_1$. Then for every $i > 1$ the set $X_i$ forms a $k$-coalition with only $X_1$. If there is no set $X \in \copart_G$ that contains $V^-$, then necessarily $\copart_G = \{X_1, X_2\}$ where $V^-_G \subseteq X_1 \cup X_2$. In this case, it is still true that every set in $\copart \setminus\{X_1\}$ forms a $k$-coalition with $X_1$. By an analogous argument, we can show that for any maximum-size $k$-coalition partition $\copart_H$ of $H$, there exists a set $Y_1$ such that every set in $\copart_H \setminus \{Y_1\}$ forms a $k$-coalition with only $Y_1$. By \Cref{Cor: disjoint union lower bound by max of k-coalition numbers}, it follows that $\COk(G \sqcup H) \geq \max\{\COk(G), \COk(H) \}$. \\
    
    \paragraph{\textbf{Upper bound.}}
    Since $k > \max\{\delta(G), \delta(H)\} \geq \min\{\delta(G), \delta(H)\}$, by \Cref{Prop: k > min degree G disjoint union H upper bound on CO_k} we have  $\COk(G \sqcup H) \leq \max\{\COk(G), \COk(H)\}$, as desired.
\end{proof}

We may directly apply this proposition to get the following nice corollary. 

 \begin{corollary}
     Let $T_1, \dots, T_m$ be trees with more than $k$ vertices and let $F$ be the forest $F = T_1 \sqcup T_2 \sqcup \dots \sqcup T_m$. Then $\COk(F) = \max\{\COk(T_1), \COk(T_2), \dots, \COk(T_m)\}$ for all $k \geq 2$.
 \end{corollary}

We now prove an analogue to \Cref{Theorem: Join Limit} for disjoint union. 

\begin{theorem} \label{Theorem: Disjoint Union limit bound}
    Fix $k,n \in \N$ such that $2 \leq k \leq n$. If $G_1, \dots, G_m$ are (possibly distinct) $n$-vertex graphs where $m \geq 2$, then $$\COk(G_1 \sqcup \dots \sqcup G_m) \leq (n-k+1)(n-k+2) \leq n^2.$$
\end{theorem}
\begin{proof}
   If $k > \Delta(G_1 \sqcup \dots \sqcup G_m)$, then by \Cref{Observation: when k > max degree}, $\COk(G_1 \sqcup \dots \sqcup G_m) = 2$, and since $2 \leq k \leq n$ it holds that $ 2\leq (n-k+1)(n-k+2)$. We may now assume $k \leq \Delta(G_1 \sqcup \dots \sqcup G_m))$.
   
    Let $\copart = \{ X_1, \dots, X_{\COk(G_1 \sqcup \dots \sqcup G_m)} \}$ be a maximum-size $k$-coalition partition of $G_1 \sqcup \dots \sqcup G_m$. Each $X_i \in \copart$ intersects $G_1$ or forms a coalition with some set $X_j$ that intersects $G_1$. 
    By \Cref{Observation: No max size partition has k-dom set}, there are no $k$-dominating sets in $\copart$. Therefore, there exist two $k$-coalition partners $X_1, X_2 \in \copart$ whose union $k$-dominates $G_1$ and $|(X_1 \cup X_2) \cap G_1| \geq k$. There are at most $n - k + 2$ sets in $\copart \setminus \{X_1,X_2\}$ that intersect $G_1$. 
  By \Cref{Lemma: max number of coalition partners}, each such set has at most $\Delta(G_1 \sqcup \dots \sqcup G_m) -k +3\leq (n-1) - k +2$ distinct $k$-coalition partners. Then
  \begin{equation*}
 \COk(G_1 \sqcup \dots \sqcup G_m) = |\copart| \leq (n -k +1) (n -k +2) \leq n^2,
  \end{equation*}
  where the last inequality follows from the fact that $k\geq 2$. 
\end{proof}

\begin{corollary}
    Fix $k,n,m \in \N$ such that $2 \leq k \leq n$ and $2\leq m$. For any $n$-vertex graph $G$, 
    $$\COk(m*G) \leq n^2,$$
    where $m*G = G \sqcup \dots \sqcup G$ denotes the disjoint union of a graph with itself $m$ times.
\end{corollary}

\section{Distinct sizes of $k$-coalition partitions} \label{sect: I_k(G)}
The previous work on $k$-coalitions by Jafari, Alikhani, and Bakhshesh \cite{JAB25} and Bre\v{s}ar, Henning, and Samadi \cite{BHS25} has focused on computing and bounding the $k$-coalition number.
Haynes, Hedetniemi, Hedetniemi, McRae, and Mohan \cite{HHHMM20} proposed the study of the minimum number of sets in a coalition partition. This was investigated by Bakhshesh and Henning \cite{MinMinCoalition}. In this section, we generalize this to the minimum $k$-coalition number, which we denote by $\cok(G)$. We also propose a new question: what are all possible sizes of the $k$-coalition partitions of a given graph? To this end, we provide some definitions.

\begin{definition}
    Let $\cok(G)$ denote the minimum number of sets in any $k$-coalition partition of $G$. Let $$\kset(G) = \{n \in \N \mid \text{there exists a $k$-coalition partition of $G$ with exactly $n$ sets} \}.$$ 
\end{definition}

By definition, $\kset(G)\subseteq[\cok(G), \COk(G)]$; the main theorem of this section will show that $\kset(G)$ is the entire interval.

\begin{theorem} \label{Theorem: I_k is an interval}
    Fix $k\geq 2$. Let $G$ be a graph. The set $\kset(G)$ is the interval $\left[\cok(G), \COk(G)\right ]$.
\end{theorem}

The proof is divided into cases depending on the relationships between $k$ and $\delta(G)$. We first prove a few results that we will require later.

\begin{proposition} \label{Prop: Min size partition of K_n}
    Let $k \leq n$. The minimum $k$-coalition partition number of $K_n$ is 
    \begin{align*}
        \cok(K_n) = \begin{cases}
            \frac{n}{k} & \text{if $k \mid n$}, \\
            \lceil \frac{n}{k} \rceil  & \text{if $k \nmid n$ and $k \nmid n+1$}, \\
             \lceil \frac{n}{k} \rceil +1   & \text{if $k \nmid n$ and $k \mid n+1$}.
        \end{cases}
    \end{align*}
\end{proposition}

\setcounter{tbox}{0}
\begin{proof}
Let $\copart = \{X_1, \dots, X_{\cok(G)}\}$ be a minimum-size $k$-coalition partition of $K_n$. Label the vertices of $K_n$ as $\{v_1, \dots, v_n\}$. For every $X_i \in \copart$, we have $|X_i| \leq k$ as otherwise $X_i$ would be a $k$-dominating set with strictly greater than $k$ vertices. By the pigeonhole principle, $|\copart| \geq \lceil \frac{n}{k} \rceil$.
There are three cases:

\sta{\label{Case: k divides n} If $k \mid n$, then $\cok(K_n) = \frac{n}{k}$.}

If $k \mid n$, then $|\copart| \geq \frac{n}{k}$. We obtain the upper bound from the $k$-coalition partition consisting of $\frac{n}{k}$ disjoint $k$-dominating sets, each with exactly $k$ vertices. This proves \eqref{Case: k divides n}. 

\sta{\label{Case: k does not divide n or n+1} If $k \nmid n$ and $k\nmid n+1$, then $\cok(K_n) = \left\lceil\frac{n}{k}\right\rceil$. }

Since we showed $\cok(K_n) \geq \left\lceil\frac{n}{k}\right\rceil$, it suffices to exhibit a $k$-coalition partition with exactly $\left\lceil\frac{n}{k}\right\rceil$ sets. Let $m = \left\lfloor \frac{n}{k} \right\rfloor$ and let $D_i = \{v_{ik+1},\dots,v_{(i+1)k} \}$. Then $D_i$ is a $k$-dominating set with exactly $k$ vertices and the following is a $k$-coalition partition: 
\begin{equation} \label{eq: Kn min partition}
    \copart = \left\{D_0, \dots, D_{m - 2}\right\} \cup \left\{\{v_{(m-1)k + 1}, \dots ,v_{mk -1}\} \right\} \cup \left\{ \{v_{mk}, \dots, v_n \} \right\}.
\end{equation}
Since $k \nmid n+1$, it must be that $ mk < n < (m+1)k-1$. Therefore, the latter two sets both have strictly fewer than $k$ vertices, while their union has at least $k$ vertices, and so these two sets form a $k$-coalition. This proves \eqref{Case: k does not divide n or n+1}.

\sta{\label{Case: k divides n+1} If $k \nmid n$ and $k\mid n+1$, then $\cok(K_n) = \left\lceil\frac{n}{k}\right\rceil +1 $. }

In this case, the partition in \eqref{eq: Kn min partition} is not a $k$-coalition partition because $k \mid n+1$ implies that $n = (m+1)k - 1 $, and so the set $\left\{ v_{mk}, \dots, v_n \right\}$ is a $k$-dominating set with $k$ vertices. Consequently, the set $\left\{v_{(m-1)k + 1}, \dots ,v_{mk -1} \right\}$ does not have a $k$-coalition partner. In fact, every partition of $V(K_n)$ into $\left\lceil \frac{n}{k} \right\rceil$ sets (where each set has no more than $k$ vertices) has exactly one set with $k-1$ vertices and no $k$-coalition partner. Consequently, $\cok(G) > \left\lceil \frac{n}{k} \right\rceil$. However, if we define $D_i$ as above and set
\begin{equation*}
     \copart = \left\{D_0, \dots, D_{m - 2}\right\} \cup \left\{ \{v_{(m-1)k + 1}, \dots, v_{mk -1}\} \right\} \cup \left\{ \{ v_{mk}, \dots, v_{n-1} \}\right\} \cup \{v_n\},
\end{equation*}
then $\copart$ is a $k$-coalition partition of $G$ into $\left\lceil \frac{n}{k} \right\rceil + 1$ sets, which proves \eqref{Case: k divides n+1}.
See \Cref{fig:K_n Lower Bound Examples} for an example of this construction.
\end{proof}

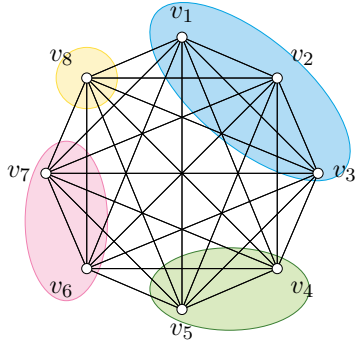
\begin{figure}[htbp]
    \centering
    \scalebox{0.9}{
    \begin{tikzpicture}[scale=2,
  vertex/.style={circle, draw, fill=white, inner sep=1.5pt}
]

\coordinate (v1) at (0,1);
\coordinate (v2) at (0.7,0.7);
\coordinate (v3) at (1,0);
\coordinate (v4) at (0.7,-0.7);
\coordinate (v5) at (0,-1);
\coordinate (v6) at (-0.7,-0.7);
\coordinate (v7) at (-1,0);
\coordinate (v8) at (-0.7,0.7);

\node[vertex, label=above:{$v_1$}] at (v1) {};
\node[vertex, label=above right:{$v_2$}] at (v2) {};
\node[vertex, label=right:{$v_3$}] at (v3) {};
\node[vertex, label=below right:{$v_4$}] at (v4) {};
\node[vertex, label=below:{$v_5$}] at (v5) {};
\node[vertex, label=below left:{$v_6$}] at (v6) {};
\node[vertex, label=left:{$v_7$}] at (v7) {};
\node[vertex, label=above left:{$v_8$}] at (v8) {};

\begin{scope} [on background layer]
    \node[ellipse, draw=OliveGreen,  fill=LimeGreen!30, fit=(v4)(v5)] {};
    \node[ellipse, draw=CarnationPink, fill=CarnationPink!30, fit=(v6)(v7)] {};
    \node[ellipse, draw=Goldenrod, fill=Goldenrod!30, minimum width = 0.9cm, minimum height = 0.9cm, fit=(v8)] {};
 \fill[Cerulean!30, rotate around = {-40: (0.5,0.6)} ] (0.5,0.6) ellipse (0.9 and 0.4);
  \draw[Cerulean, rotate around = {-40: (0.5,0.6)} ] (0.5,0.6) ellipse (0.9 and 0.4);

\end{scope}

\begin{scope}[on background layer]
  \foreach \i in {1,...,8} {
    \foreach \j in {1,...,8} {
      \draw (v\i) -- (v\j);
    }
  }
\end{scope}
\end{tikzpicture}}
    \caption{Example of a minimum-cardinality $3$-coalition partition of $K_8$, where the sets in the partition are represented by the ellipses.}
    \label{fig:K_n Lower Bound Examples}
\end{figure}

The only case where a $k$-coalition partition can consist of exactly one set occurs when the graph $G$ has exactly $k$ vertices. The following lemma handles this edge case.
\begin{lemma} \label{Lemma: Ik includes 1 only when G has k vertices}
    If $|V(G)| = k \geq 2$, then $\kset(G) = [1, 2]$.
\end{lemma}

\begin{proof}
    If $|V(G)| = k$, then $\Delta(G) < k$. By \Cref{Observation: when k > max degree} $\COk(G) = 2$ and, in particular, there exists a $k$-coalition partition with two sets. In addition, $\{V(G)\}$ is a $k$-coalition partition with one set. 
\end{proof}

The following results address the three cases: $k > \delta(G)$, $k= \delta(G)$, and $k < \delta(G)$.
When $k > \delta(G)$, there exists at least one vertex $v$ with degree less than $k$ and every $k$-dominating set of $G$ must contain $v$.
\begin{lemma} \label{Lemma: interval for k > delta (G)}
    Fix $k\geq 2$. If $k > \delta(G)$ and $|V(G)| \neq k$, then $\kset(G) = [2, \COk(G)]$.
\end{lemma}

\begin{proof} \setcounter{tbox}{0}
  The strategy is to fix a maximum-size $k$-coalition partition and then carefully merge sets such that the resulting partition is still a $k$-coalition partition.
  Let $V^- = \{v \in V \mid \deg(v) <  k\}$.  Let $\copart = \{X_1, X_2, \dots, X_{\COk(G)}\}$ be a maximum-size $k$-coalition partition. Suppose there does not exist a set $X\in \copart$ such that $V^- \subseteq X$. Then $\copart = \{X_1, X_2\}$, where $X_1$ and $X_2$ are $k$-coalition partners and $V^- \subseteq X_1 \cup X_2$. In this case, $\COk(G) = 2$. 

We may now assume there exists some set $X\in\copart$ such that $V^-\subseteq X$; without loss of generality, say $V^-\subseteq X_1$.
For every $i > 1$, the set $X_i$ is not a $k$-dominating set (since it does not contain $V^-$). Moreover, $X_i$ can form a $k$-coalition with only $X_1$. So, for every $n \in [2, \COk(G)]$, let 
         \begin{equation*}
             \copart_n = \left\{X_1, X_2, \dots, X_{n-1}, X_n \cup \dots \cup X_{\COk(G)} \right\}.
         \end{equation*}
         For every $n \geq 2$, the set $X_n \cup \dots \cup X_{\COk(G)}$ is not $k$-dominating (since it does not contain $V^-$) and forms a $k$-coalition with $X_1$. Hence, $\copart_n$ is a $k$-coalition partition of $G$ with exactly $n$ sets.
\end{proof}

We next consider the case where $k = \delta(G)$. If $v$ is a vertex of minimum degree, we observe that every $k$-dominating set contains either $v$ or the entire neighborhood of $v$.

\begin{lemma} \label{Lemma: interval for k = delta (G)}
    Let $k = \delta(G)$ and $k \geq 2$. If $k = 2$ and $G = K_3$,  then $\kset(G)= \{3\}$. Otherwise, $\kset(G) = [2, \COk(G)]$.
\end{lemma}

\begin{proof}
Let $v \in V(G)$ be a vertex of minimum degree. If every vertex $u \in V(G) \setminus \{v\}$ is adjacent to $v$, then $|V(G)| = k+1$ and every vertex has degree $k$. Therefore, $G$ is the clique $K_{k+1}$. By \Cref{Thm: CO_k(K_n)}, $\COk(K_{k+1}) = (k+1) - k + 2 = 3$. If $k =2$, then $G = K_3$. There is no $2$-coalition partition of $K_3$ with two sets, 
since if we partition $V(K_3)$ into two sets, then one set contains one vertex, say $X_1$, and one set contains two vertices, say $X_2$. However, $X_1$ has no $2$-coalition partner because $X_2$ is a $2$-dominating set with $2$ vertices. Therefore, the only $2$-coalition partition of $K_3$ is the partition into three singleton sets, and so $I_2(K_3) = \{ 3\}$. If $k > 2$, then $k \nmid (k+1)+1$. By \Cref{Prop: Min size partition of K_n}, $\cok(K_{k+1}) = \left\lceil \frac{k+1}{k} \right\rceil = 2$. We conclude that $\kset(K_{k+1}) = [2,3]$ for all $k\geq 3$.

We may now assume that $V(G) \setminus N[v] \neq \emptyset$. We first construct a $k$-coalition partition with two sets. Pick some $u \in N(v)$ and set $\copart =  \{V(G) \setminus \{v, u\},\{v, u\} \}$. The set $V(G) \setminus \{v, u\}$ does not $k$-dominate $v$ because it neither contains $v$ nor the entire neighborhood of $v$. The set $\{v, u\}$ does not $k$-dominate the vertices in $V(G) \setminus N[v]$ since $k \geq 2$. However, the union of these sets $k$-dominates $G$. Therefore, $\copart$ is a $k$-coalition partition.

It just remains to construct $k$-coalition partitions with $n$ sets, where $n\in [3,\COk(G)]$ and $\COk(G) \geq 3$. Fix a maximum-size $k$-coalition partition $\copart$ of $G$ and define
\begin{align*}
    \pi_1 = \{ X \in \copart \mid X\cap N[v] \neq \emptyset \}; \\
    \pi_2 = \{ X \in \copart \mid X\cap N[v] = \emptyset \}.
\end{align*}
Set $p_1 = |\pi_1|$ and $p_2 = |\pi_2|$; it may be that $\pi_2 = \emptyset$. Then $\COk(G) = p_1 + p_2$. Label the sets in $\pi_1$ as $X_1, \dots, X_{p_1}$ and the sets in $\pi_2$ as $Y_1, \dots, Y_{p_2}$, such that $X_1$ is the set containing $v$. 

For each $n \in [3, p_1]$, we will construct a $k$-coalition partition with $n$ sets.
First suppose $p_1 \geq 3$.
Then $X_2$ must contain some neighbor of $v$ but not all of $N(v)$. Let $Y = Y_1 \cup Y_2 \cup \dots \cup Y_{p_2}$. Note that if $\pi_2 = \emptyset$, then $Y = \emptyset$. We verify that the following is a $k$-coalition partition with $n$ sets:
\begin{equation*}
    \{ X_1, X_2, \dots, X_{n-1}\} \cup \{X_n \cup \dots \cup X_{p_1} \cup Y \}.
\end{equation*}
Indeed, the set $X_n \cup \dots \cup X_{p_1} \cup Y $ is not a $k$-dominating set because it contains at most $k-1$ neighbors of $v$ and does not $k$-dominate $v$. Since the union of $X_{p_1}$ with any of the sets $X_3, \dots, X_{p_1 - 1}, Y$ does not form a $k$-dominating set (since it cannot $k$-dominate $v$), the $k$-coalition partner of $X_{p_1}$ in the initial partition must have been $X_1$ or $X_2$. Therefore, $X_n \cup \dots \cup X_{p_1} \cup Y $ forms a $k$-coalition with $X_1$ or $X_2$.

If $p_2 = 0$, we are done because $p_1 = \COk(G)$, where we assumed $\COk(G) \geq 3$. Suppose $\pi_2 \neq \emptyset$. Then, for every $n\in[p_1 +1 , \COk(G)]$,  the following is a $k$-coalition partition with $n$ sets:
\begin{equation*}
    \pi_1 \cup \{Y_1, \dots, Y_{n-p_1-1} \} \cup  \{ Y_{n-p_1} \cup \dots \cup  Y_{p_2} \}.
\end{equation*}
Indeed, the set $Y_{n-p_1} \cup \dots \cup  Y_{p_2}$ is not a $k$-dominating set because it does not $k$-dominate $v$ (since it contains no neighbors of $v$). However, this set forms a $k$-coalition with the $k$-coalition partner(s) of $Y_{n-p_1}, \dots, Y_{p_2}$ in $\pi_1$. Note that this construction shows that if $p_1 = 1$ then $[2,\COk(G)] \subseteq \kset(G)$ and if $p_1 = 2$ then $[3,\COk(G)] \subseteq \kset(G)$, as required.
\end{proof}

The $k < \delta(G)$ case is the hardest and the proof requires two approaches. First, we begin with a maximum-size $k$-coalition partition of $G$ and strategically merge sets one-by-one to gradually decrease the number of sets, as we did in the previous proof. Second, we start with a minimum-size $k$-coalition partition, then carefully rearrange the sets and split off singletons one-by-one to gradually increase the number of sets. 

\begin{lemma} \label{Lemma: upper interval for k < delta (G)}
    Fix $k\geq 2$. If $k < \delta(G)$, then $[\delta(G) - k +4, \COk(G)] \subseteq \kset(G)$.
\end{lemma}

\begin{proof} \setcounter{tbox}{0}
    Let $\copart$ be a maximum-size $k$-coalition partition of $G$ and let $v$ be a vertex of minimum degree. By \Cref{Observation: No max size partition has k-dom set}, $\copart$ does not contain any $k$-dominating sets. Define
    \begin{align*}
    \pi_1 = \{ X \in \copart \mid X\cap N[v] \neq \emptyset \}; \\
    \pi_2 = \{ X \in \copart \mid X\cap N[v] = \emptyset \}.
\end{align*}
As before, let $p_1 = |\pi_1|$ and $p_2 = |\pi_2|$. Label the sets of $\pi_1$  and $\pi_2$ as $X_1, \dots, X_{p_1}$ and $Y_1, \dots, Y_{p_2}$, respectively. When $\pi_2 \neq \emptyset$, for every $n \in [p_1+1, \COk(G)]$, the following is a $k$-coalition partition with $n$ sets:
\begin{equation*}
    \pi_1 \cup \{Y_1, \dots, Y_{n-p_1-1} \} \cup  \{ Y_{n-p_1} \cup \dots \cup  Y_{p_2} \}.
\end{equation*}

Let $Y = Y_1 \cup \dots \cup Y_{p_2}$. Let $\copart' = \{X_1, X_2, \dots, X_{p_1}\} \cup \{Y\}$. Note that $\copart'$ does not contain a $k$-dominating set, since $\copart$ did not contain a $k$-dominating set and the new set $Y$ does not $k$-dominate $v$.
If there exist $X, X' \in \copart'$ that are not $k$-coalition partners, then replace $\copart'$ with $\copart'' = (\copart' \setminus \{X, X'\} ) \cup \{X \cup X'\}$. We call this \textit{merging} two sets. 
Observe that $X \cup X'$ is not a $k$-dominating set (otherwise, $X$ and $X'$ would have been $k$-coalition partners), but $X \cup X'$ forms a $k$-coalition with the $k$-coalition partners of $X$ and $X'$. And, if some set $X''$ formed a coalition with $X$ or $X'$, it now forms a coalition with $X \cup X'$. Therefore, we conclude that $\copart''$ is a $k$-coalition partition, with one fewer set than $\copart'$. Repeat until no suitable $X$ and $X'$ remain. 

Let $\copart ^\star = \{ X_1, \dots, X_r, Y\}$ be the resulting $k$-coalition partition. Through the merging process, we have constructed a $k$-coalition partition with $n$ sets for every $ |\copart^\star| < n < |\copart'|$.
Let $X_1$ denote the set containing $v$ and let $Y$ denote the set that does not intersect $N[v]$, if such a set exists.

\sta{\label{Claim: size upper bound} It holds that $|\copart^\star| \leq \left\lceil \frac{2\delta(G)}{k} \right\rceil +2$.
}

Let $\pi_3 = \copart^\star \setminus \{X_1, Y\}$, then every set in $\pi_3$ intersects $N(v)$. Set $|\pi_3| = p_3$.
If there exist two sets $X, X' \in\pi_3$ that both contain fewer than $\frac{k}{2}$ neighbors of $v$, then their union $X \cup X'$ does not $k$-dominate $v$. Hence, $X$ and $X'$ are not $k$-coalition partners and could be merged.  However, since the merging process has terminated, there is at most one set $X \in  \pi_3$ with $|N(v) \cap X| < \frac{k}{2} $. 

Suppose, for contradiction, $|\copart^\star| \geq  \left\lceil \frac{2\delta(G)}{k} \right\rceil +3$, then $p_3 \geq \left\lceil \frac{2\delta(G)}{k} \right\rceil +1.$ 
If $k = 2$, then $p_3 > \delta(G)$, which contradicts the fact that every set in $\pi_3$ intersects $N(v)$. We may now assume $k \geq 3$. All but one set, say $X \in \pi_3$, contain at least $\left\lceil \frac{k}{2}\right\rceil$ neighbors of $v$. Note that $X$ has nonzero intersection with $N[v]$, so it follows that
\begin{align*}
    \delta(G) 
    \geq \left\lceil \frac{2\delta(G)}{k} \right\rceil \cdot  \left\lceil \frac{k}{2} \right\rceil + |X\cap N(v)|
    \geq \delta(G) +  |X\cap N(v)|> \delta(G).
\end{align*}
This is a contradiction, which proves \eqref{Claim: size upper bound}.\\

We have now constructed a $k$-coalition partition with at most $\left\lceil \frac{2\delta(G)}{k} \right\rceil + 2$ sets. It is sufficient to observe that for all $2 \leq k < \delta(G)$ we have
\begin{align*}
    \left\lceil \frac{2\delta(G)}{k} \right\rceil + 2 
    =   \left\lceil 2 + \frac{2}{k} \cdot (\delta(G)-k) \right\rceil + 2 
    \leq  \big\lceil 2+ \delta (G) -k \big\rceil + 2
    =   \delta (G) -k + 4,
\end{align*}
concluding the proof.
\end{proof}



\begin{lemma} \label{Lemma: lower interval for k < delta (G)}
     Fix $k \geq 2$. If $k < \delta(G)$, then $[\cok(G), \delta(G) - k +3] \subseteq \kset(G)$.
\end{lemma}

\begin{proof} \setcounter{tbox}{0}
    Fix a minimum size $k$-coalition partition $\copart = \{X_1, X_2, \dots, X_{\cok(G)}\}$. If every set in $\copart$ is a $k$-dominating set, then split $X_1$ into two non-empty sets $X_1$ and $X_{\cok(G) + 1}$ that form a $k$-coalition pair. We now may suppose that at least one set in $\copart$ is not a $k$-dominating set; without loss of generality, $X_1$ is such a set. Let $m$ be the number of sets in the smallest $k$-coalition partition with at least one $k$-coalition pair. By definition, $\cok(G) \leq m \leq \cok(G) + 1$.

    If there exists a vertex $u \in X_i$, for some $i \neq 1$, such that $\{u\} \cup X_1$ is not a $k$-dominating set, then replace $X_1$ by $X_1 \cup \{u\}$ and replace $X_i$ by $X_i \setminus\{u\}$. Denote the new $k$-coalition partition by $\copart'$. Repeat this step with $\copart'$ until no such vertex can be found. Let $\copart''$ be the resulting partition.

    \sta{\label{Claim: X'' is a k-coalition partitions} The collection $\copart''$ is a $k$-coalition partition of $G$ with at most $m$ sets.}
    At every step in the process, the number of sets in the partition either stays the same or decreases, hence $|\copart''| \leq |\copart| =m$. With the exception of $X_1$, no other set gains vertices. Therefore, if a set was not previously a $k$-dominating set with $k$ vertices, it did not become one. Second, when the process terminates, the set $\{ u\} \cup X_1$ is a $k$-dominating set for every $u\in V(G) \setminus X_1$. Therefore, for each $X_i$ where $i \neq 1$, one of the following holds: either $X_i$ is a $k$-dominating set with exactly $k$ vertices or $X_i$ forms a $k$-coalition with $X_1$. This proves \eqref{Claim: X'' is a k-coalition partitions}.
    \vspace{3mm}

    Given that $X_1$ is not a $k$-dominating set, there exists at least one vertex, say $v \in V(G) \setminus X_1$, that is not $k$-dominated by $X_1$. We make two structural observations.\\

\sta{\label{Claim: structural observations} If $v \in V(G) \setminus X_1$ is not $k$-dominated by $X_1$, then $|N(v) \cap X_1 | = k-1$ and  $V(G) \setminus N[v] \subseteq X_1$.}

Since $X_1$ does not $k$-dominate $v$, it must be that $|N(v) \cap X_1 | \leq k-1$. However, if $X_1$ has strictly fewer than $k-1$ neighbors of $v$, then for any $v' \in N(v) \cap (V(G) \setminus X_1)$, the set $\{v'\} \cup X_1$ is not a $k$-dominating set, contradicting the assumption that the process terminated. This proves the first claim.
    
To prove the second claim, assume otherwise that there exists some vertex $u \notin X_1$ and $u \in V(G) \setminus N[v] $. Then $\{u\} \cup X_1$ would not $k$-dominate $v$, contradicting the assumption that the process had terminated. This proves \eqref{Claim: structural observations}.
\vspace{3mm}

 Now, we have the $k$-coalition partition $\copart'' = \{X_1, \dots, X_m\}$, where $X_i \subseteq N[v]$ for all $i\geq 2$. We define a \textit{splitting} process analogously to the merging process. If there is a set $X \in \copart'' \setminus \{X_1\}$ where $|X| \geq 2$, then split $X$ into two sets: a singleton set $\{u\}$ and $X\setminus \{u\}$. Set $\copart''' = \big( \copart'' \setminus \{X\}  \big) \cup \{u\} \cup  X\setminus \{u\}$.

We verify that $\copart'' $ is a $k$-coalition after the splitting step. Since every vertex in $X$ forms a $k$-coalition with $X_1$, clearly both $\{u\}$ and $X\setminus \{u\}$ form a $k$-coalition with $X$. Even if $X$ is a $k$-dominating set with $k$-vertices, the set $X\setminus \{u\}$ has fewer than $k$ vertices and cannot $k$-dominate $G$. 
Repeat the splitting step until every set in the partition, except $X_1$, is a singleton set. Let $\copart^\star$ denote the resulting partition. The number of sets in $\copart^\star$ is at least 
    \begin{equation*}
        |\copart^\star| = (\deg(v) + 1) - (k-1) + 1 \geq \delta(G) -k +3.
    \end{equation*}
    In the course of this splitting process, we obtain $k$-coalition partitions of all intermediate sizes between $\cok(G)$ and $|\copart''|$. We conclude that $[\cok(G), \delta(G)-k+3] \subseteq \kset(G)$.
\end{proof}

\begin{proof} [Proof of \Cref{Theorem: I_k is an interval}]

Fix $k \geq 2$.
If the graph $G$ has exactly $k$ vertices, then by \Cref{Lemma: Ik includes 1 only when G has k vertices} $\kset(G) = [1, 2]$. Otherwise, $\cok(G) \geq 2$.
If $k \geq \delta(G)$, then by \Cref{Lemma: interval for k > delta (G)} and \Cref{Lemma: interval for k = delta (G)}, there are only two cases.
\begin{enumerate}
    \item If $k =2$ and $G = K_3$, then $I_2(K_3) = \{3\}$.
    \item Otherwise, $\kset(G) = [2, \COk(G)]$. 
\end{enumerate}
In either case, we have $\kset(G) = [\cok(G), \COk(G)]$.
If $k < \delta(G)$, then by \Cref{Lemma: upper interval for k < delta (G)} and \Cref{Lemma: lower interval for k < delta (G)}, it follows that $[\delta(G) - k + 4, \COk(G)] \subseteq \kset(G)$ and $[\cok(G), \delta(G) - k + 3] \subseteq \kset(G)$. Therefore, $\kset(G) = [\cok(G), \COk(G)]$.
\end{proof}

We conclude by providing necessary and sufficient conditions on $G$ to obtain that $\cok(G) = 2$.
\begin{theorem}
    Fix $k\geq 2$. Suppose $G$ is a non-complete graph with more than $k$ vertices. Then $\cok(G) = 2$ if and only if there exist two vertices $u$ and $v$ such that $|N[u] \cap N[v]| \leq 2k-2$.
\end{theorem}

\begin{proof}
    First, suppose that there exists a $k$-coalition partition of $G$ with two sets, say $\copart = \{ X_1, X_2\}$. We will find two vertices $u$ and $v$ with $|N[u] \cap N[v]| \leq 2k-2$. We consider two cases: \\
    \textbf{Case 1:} If $X_1$ is a $k$-dominating set with $k$ vertices, then $X_2$ is also a $k$-dominating set with $k$ vertices. So $|V(G)| = | X_1| + |X_2| = 2k$. Given that $G$ is not a complete graph, there exist at least two nonadjacent vertices $v$ and $u$. 
    Since $u, v \notin N[u] \cap N[v]$, we have $|N[u] \cap N[v]| \leq 2k-2$.
    
    \noindent\textbf{Case 2:} Suppose that $X_1$ is not a $k$-dominating set. Then $X_2$ is not a $k$-dominating set. Therefore, there is at least one vertex $v \in X_2$ that is not $k$-dominated by $X_1$ and at least one vertex $u \in X_1$ that is not $k$-dominated by $X_2$. This implies
    \begin{equation*}
        |N[v] \cap X_1|  \leq k-1 \text{ and } |N[u] \cap X_2|\leq k-1.
    \end{equation*} Then it follows that 
        \begin{equation*}
            \big|N[u] \cap N[v]\big| = \big|N[u] \cap N[v] \cap X_1\big|+\big|N[u] \cap N[v] \cap X_2\big| \leq 2k-2.
        \end{equation*} 

Now, to prove the rest of the statement, assume that there exist two vertices $u$ and $v$ such that $|N[u] \cap N[v]| \leq 2k-2$. 
We will construct a $k$-coalition partition with exactly two sets. 

First suppose $\min(\deg(u),\deg(v)) \leq k-2$; without loss of generality $\deg(v) \leq k-2$. Then $\{N[v], G \setminus N[v]\}$ is a $k$-coalition partition of $G$. Indeed, $G\setminus N[v]$ cannot $k$-dominate $v$ and, since $|N[v]|\leq k-1$, does not $k$-dominate any vertex in $G \setminus N[v]$ (which is non-empty because $|V(G)| \geq k$). We may now assume both $\deg(u), \deg(v) \geq k-1$. Let $\deg(v) = d$ and label the vertices of $N(v)$ as $\{v_1, \dots, v_d\}$. If $u$ and $v$ are adjacent, relabel such that $v_1 = u$. If $|N[u] \cap N(v)| \leq k-1$, suppose that $N[u] \cap N(v) \subseteq \{v_1, \dots , v_{k-1}\}$. Otherwise, let $\{v_1, \dots, v_{k-1}\} \subseteq N(v) \cap N[u]$.  Let 
    \begin{align*}
        &X_1 = \left( V(G) \setminus N[v] \right) \cup \{v_1, \dots, v_{k-1}\}; \\
       & X_2 = \{v,v_k,v_{k+1},\dots, v_d\}.
    \end{align*}
Observe that $v \in X_2$ is not $k$-dominated by $X_1$, which contains at most $k-1$ neighbors of $v$ and $u \in X_1$ is not $k$-dominated by $X_2$ which contains at most $k-1$ neighbors of $u$. Therefore, $\{X_1, X_2\}$ is a $k$-coalition partition. 
\end{proof}

\begin{corollary}
Fix $k \geq 2$. If $k< \delta(G)$, then any one of the following conditions is sufficient to ensure that $G$ has $\cok(G) = 2$.
\begin{enumerate}
    \item $G$ is not connected;
    \item $\operatorname{diam}(G) \geq 3; $
    \item $\delta(G) \leq 2k - 2.$
\end{enumerate}
\end{corollary}

\section{$k$-coalition graphs} \label{sect: k-coalition partition graphs}

In this section, we study \textit{$k$-coalition graphs} as proposed by Jafari, Alikhani, and Bakhshesh \cite{JAB25}, which generalize the coalition graphs initially proposed in \cite{HHHMM20} and further studied in \cite{HHHMM23CCO,HHHMM23DMGT,HHHMM23OM}.
We provide a structural characterization of the $k$-coalition graph when $k \leq \delta(G)$ and we show that every graph admits a $k$-coalition partition such that the resulting $k$-coalition graph is a complete graph. The main theorem of this section establishes that every graph is the $k$-coalition graph of some other graph. The result was shown for $k=1$ by Haynes, Hedetniemi, Hedetniemi, McRae, and Mohan in \cite{HHHMM23CCO}. Their proof proceeds by explicit construction, and we generalize it for all $k$.

\begin{definition}
    Let $G$ be a graph. Let $\copart = \{X_1,\dots, X_{|\copart|}\}$ be a $k$-coalition partition of $G$. The \textit{$k$-coalition graph}, denoted by $\kcograph(\copart, G)$, is the graph with vertex set $\{x_1, \dots, x_{|\copart|}\}$, where two vertices $x_i$ and $x_j$ are adjacent if and only if $X_i$ and $X_j$ form a $k$-coalition. By definition, if $X_i \in \copart$ is a $k$-dominating set, then $x_i$ is an isolated vertex in $\kcograph(\copart, G)$. See \Cref{fig: k-coalition partition example} for an example.
\end{definition}

    \begin{figure} [htb]
    \centering
    \begin{tikzpicture}[
  every node/.style={circle, draw, fill=white, inner sep=1.5pt}, 
  scale=1
]

\node (A1) at (1,0) {};
\node (A2) at (2,0) {};
\node (A3) at (3,0) {};
\node (A4) at (4,0) {};

\node (B1) at (1, 2) {};
\node (B2) at (2, 2) {};
\node (B3) at (3, 2) {};
\node (B4) at (4, 2) {};

\foreach \i in {A1,A2,A3, A4}
  \foreach \j in {B1,B2,B3,B4}
    \draw(\i) -- (\j);

\begin{scope} [on background layer]
    \node[ellipse, draw=Cerulean, fill=Cerulean!30, minimum width=0.5cm, minimum height=0.5cm, fit=(A2)] {};
    \node[ellipse, draw=Cerulean, fill=Cerulean!30, minimum width=0.5cm, minimum height=0.5cm, fit=(B1)(B2)] {}; 
    \node[ellipse, draw=CarnationPink, fill=CarnationPink!30, minimum width=0.5cm, minimum height=0.5cm, fit=(A1)] {};
    
     \node[ellipse, draw=Goldenrod, fill=Goldenrod!30, minimum width=0.5cm, minimum height=0.5cm, fit=(B4)] {};
     \node[ellipse, draw=OliveGreen,  fill=LimeGreen!30, minimum width=0.5cm, minimum height=0.5cm, fit=(A3)(A4)] {};
    \node[ellipse, draw=OliveGreen,  fill=LimeGreen!30, minimum width=0.5cm, minimum height=0.5cm, fit=(B3)] {};
\end{scope}

\draw[->](4.5,1)   -- (5.5,1);

\node (C1) at (6,1) {};
\node (C2) at (7,1) {};
\node (C3) at (8,1) {};
\node (C4) at (9,1) {};

\begin{scope} [on background layer]
    \node[ellipse, draw=CarnationPink, fill=CarnationPink!30, minimum width=0.5cm, minimum height=0.5cm, fit=(C1)] {};
    \node[ellipse, draw=Cerulean, fill=Cerulean!30, minimum width=0.5cm, minimum height=0.5cm, fit=(C2)] {};
    \node[ellipse, draw=OliveGreen,  fill=LimeGreen!30, minimum width=0.5cm, minimum height=0.5cm, fit=(C3)] {};
     \node[ellipse, draw=Goldenrod, fill=Goldenrod!30, minimum width=0.5cm, minimum height=0.5cm, fit=(C4)] {};
\end{scope}

\draw (C1) -- (C2);
\draw (C2) -- (C3);
\draw (C3) -- (C4);

\end{tikzpicture}
\caption{ (Left) The complete bipartite graph $K_{4,4}$ with the $2$-coalition partition $\copart = \{ \text{blue, green, pink, yellow}\}$. (Right) The corresponding $k$-coalition graph $\kcograph(\copart,K_{4,4})$.}
    \label{fig: k-coalition partition example}
\end{figure}
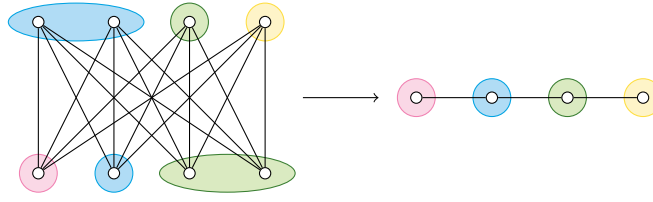

\begin{proposition}
    Let $G$ be a graph and fix $k > \delta(G)$. If $\copart$ is a $k$-coalition partition of $G$ with $n$ sets, then $\kcograph(\copart, G)$ is the star $S_{n-1}$.
\end{proposition}

\begin{proof}
    Let $\copart = \{X_1, \dots, X_{n}\}$ be a $k$-coalition partition of $G$ and let $V^- = \{ v \in V(G) \mid \deg(v) < k \}$. If $D$ is a $k$-dominating set of $G$, then $V^- \subseteq D$. 
    Suppose that there is no set $X$ in $\copart$ such that $V^- \subseteq X$. Then  $\copart = \{ X_1, X_2\}$, where $V^- \subseteq X_1 \cup X_2$. Since $X_1$ and $X_2$ are $k$-coalition partners, $\kcograph(\copart, G) = S_1$. We may now suppose that (relabeling if necessary) $V^- \subseteq X_1$. Then, for every $i > 1$, the set $X_i$ forms a $k$-coalition with $X_1$ and does not form a $k$-coalition with any set $X_j$ where $j \neq 1$. Suppose $x_i$ is the vertex in $\kcograph(\copart,G)$ corresponding to the set $X_i$. Then two vertices $x_i$ and $x_j$ with $i < j$ are adjacent if and only if $i=1$. Thus, $\kcograph(\copart, G) = S_{n-1}$. 
\end{proof}

\begin{definition} \label{Def: Modified Bistar}
    A \defn{modified bistar} is a graph with vertex set $\{v_1, v_2, \dots, v_{|V(G)|}\}$ where $v_1$ and $v_2$ are \textit{centers} and every vertex $v_i$, for $i \geq 3$, is adjacent to $v_1$, $v_2$, or both. Moreover, $v_i$ is not adjacent to $v_j$ for any $j \geq 3$. This graph is not necessarily connected. See \Cref{fig: Modified Bistar Example} for an example. 
\end{definition}

\begin{proposition} \label{Prop: kcograph k = delta(G)}
    Suppose $k = \delta(G)$ and $\copart$ is a $k$-coalition partition of $G$ with $n \geq 2$ sets. Then the $k$-coalition graph $\kcograph(\copart, G)$ has one of the following two structures: 
    \begin{enumerate}
        \item Star $S_{n-1}$;
        \item Modified bistar.
    \end{enumerate}
\end{proposition}

\begin{proof} \setcounter{tbox}{0}
    Let $\copart$ be a $k$-coalition partition with sets $X_1,  \dots, X_n$ and let $x_i$ be the vertex in $\kcograph(\copart, G)$ corresponding to $X_i$. Let $v \in V(G)$ be a vertex of minimum degree. Relabeling if necessary, suppose that $v \in X_1$.
   Since $v$ has degree $k$, any $k$-dominating set must either contain $v$ or contain the entire neighborhood $N(v)$. 

\sta{\label{kCG Case 1} If $X_1$ is a $k$-dominating set with exactly $k$ vertices, then $\kcograph(\copart, G)$ is a modified bistar. }
   
Since $|V(G)| \geq |N[v]| = k+1 > |X_1|$, there are is at least one other set in $\copart$. Since the union of the remaining sets must $k$-dominate $v$ (without containing $v$), we have $N(v) \cap X_1=\emptyset$. Relabeling if necessary, suppose that $X_2$ intersects $N(v)$. If $X_2$ is a $k$-dominating set, then there are no other sets in $\copart$ since the remaining sets could not $k$-dominate $v$. In this case, $\copart = \{X_1,X_2\}$ and $\kcograph(\copart, G)$ is the star $S_1$. If $X_2$ is not a $k$-dominating set, then every set $X$ in $\copart \setminus \{X_1,X_2\}$ can only form a $k$-coalition with $X_2$. Therefore, in $\kcograph(\copart, G)$, every vertex $x_i$ for $i \geq 3$ is adjacent only to $x_2$ while $x_1$ is a singleton. This proves \eqref{kCG Case 1}.

\sta{\label{kCG Case 2} If $X_1$ is not a $k$-dominating set, then $\kcograph(\copart,G)$ is a modified bistar.}
First, suppose that for each $i \geq 2$ the set $X_i$ forms a $k$-coalition with only $X_1$.  
 Then every vertex $x_i$ in $\kcograph(\copart, G)$ is only adjacent to $x_1$. Therefore, $\kcograph(\copart, G)$ is the star $S_{n-1}$.

Now, suppose that there exists at least one set that does not form a $k$-coalition with $X_1$. Then, there must be a set $X_2$ that intersects $N(v)$. 
This implies that every set $X_i \in \copart$ for $i \geq 3$ must form a $k$-coalition with $X_1$ (then the union contains $v$) or $X_2$ (then the union contains the entire neighborhood of $v$). For any $i, j \geq 3$, the sets $X_i$ and $X_j$ do not form a $k$-coalition because $X_i \cup X_j$ does not $k$-dominate $v$. If $X_2$ is a $k$-dominating set, then every other set must form a $k$-coalition with only $X_1$. Therefore, in $\kcograph(\copart,G)$, every vertex $x_i$ for $i \geq 3$ is adjacent to $x_1$ or $x_2$, and not adjacent to any $x_j$ with $j \geq 3$. We conclude that $\kcograph(\copart, G)$ is a modified bistar. 
\end{proof}

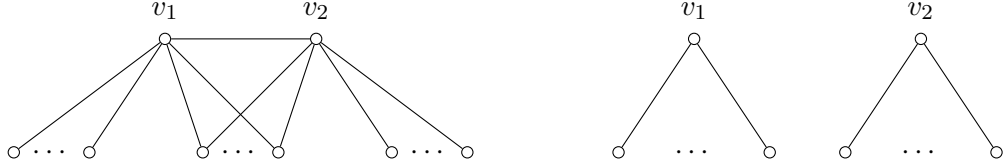
\begin{figure} [htb]
    \centering
        \begin{tikzpicture}[every node/.style={draw, circle, inner sep=1.5pt}, scale=1]

\node[label=above:$v_1$] (v1) at (0,0) {};
\node[label=above:$v_2$] (v2) at (2,0) {};

\node (A1) at (-2,-1.5) {};
\node (A2) at (-1,-1.5) {};
\node [fill = none, draw = none] at (-1.5,-1.5) {$\cdots$};

\node (B1) at (0.5,-1.5) {};
\node (B2) at (1.5,-1.5) {};
\node [fill = none, draw = none] at (1,-1.5) {$\cdots$};

\node (C1) at (3,-1.5) {};
\node (C2) at (4,-1.5) {};
\node  [fill = none, draw = none] at (3.5,-1.5) {$\cdots$};

\foreach \x in {A1, A2, B1, B2}
  \draw (v1) -- (\x);

\foreach \x in {B1, B2, C1, C2}
  \draw (v2) -- (\x);

\draw (v1) -- (v2);


\node[label=above:$v_1$] (x1) at (7,0) {};
\node[label=above:$v_2$] (x2) at (10,0) {};

\node (a1) at (6,-1.5) {};
\node (a2) at (8,-1.5) {};
\node [fill = none, draw = none] at (7,-1.5) {$\cdots$};


\node (c1) at (9,-1.5) {};
\node (c2) at (11,-1.5) {};
\node  [fill = none, draw = none] at (10,-1.5) {$\cdots$};

\foreach \y in {a1, a2}
  \draw (x1) -- (\y);

\foreach \y in {c1, c2}
  \draw (x2) -- (\y);

\end{tikzpicture}
    \caption{Two examples of a modified bistar (see \Cref{Def: Modified Bistar}).}
    \label{fig: Modified Bistar Example}
\end{figure}

\begin{proposition}
    Every graph $G$ has some $k$-coalition partition $\copart$ such that $\kcograph(\copart, G)$ is a complete graph. 
\end{proposition}
    
\begin{proof}
    Let $\copart = \{X_1, \dots, X_n\}$ be a maximum-size $k$-coalition partition of $G$.  By \Cref{Observation: No max size partition has k-dom set}, there are no $k$-dominating sets in $\copart$. 
    We now perform a merging process. 
    If there exist two sets $X,X' \in \copart$ that are not $k$-coalition partners, let $\copart' = (\copart \setminus \{X,X'\} ) \cup \{X\cup X'\}$. Since neither $X$ nor $X'$ individually $k$-dominates $G$, and they do not form a $k$-coalition, it follows that $X \cup X'$ is not a $k$-dominating set. Thus, $\copart'$ contains no $k$-dominating sets.
     In fact, $\copart$ is still a $k$-coalition partition because any set that previously formed a $k$-coalition with $X$ or $X'$ will now form a $k$-coalition with $X \cup X'$. Repeat the merging step until all pairs of sets in the partition are $k$-coalition partners. Let $\pi= \{ Y_1, \dots, Y_m\}$ denote the resulting partition. Since every two sets $Y_i$ and $Y_j$ form a $k$-coalition, every two vertices $y_i$ and $y_j$ in $\kcograph(\pi,G)$ are adjacent, as required.
\end{proof}

\begin{theorem} \label{Theorem: every graph is a k-coalition partition graph}
    For every graph $H$ and every $k\geq 2$, there exists a graph $G$ and some $k$-coalition partition $\copart$ of $G$, such that $\kcograph(\copart, G) = H$.
\end{theorem}

\begin{proof}
  The proof follows by explicit construction.  See \Cref{fig: k-coalition partition example} for an example. The $k=1$ case is handled in \cite{HHHMM23CCO} and we generalize the construction to any $k \geq 2$.
  Write $V(H) = \{ v_1, \dots, v_n\} \cup \{u_1, \dots, u_m\}$, where every vertex $v_i$ has degree at least $1$ and every $u_i$ is an isolated vertex. Let $H'$ be the subgraph induced by $V(H) \setminus \{u_1, \dots, u_m\}$. Note that $E(H') = E(H)$. 

    We begin the construction of $G$ with the complete $n$-partite graph, where every partite set has $k$ vertices. The vertex set is initially $V(K_{k, \dots, k}) = \bigcup_{i=1}^{n}\{ x_{i,1},x_{i,2}, \dots,x_{i,k} \}$, where $x_{i,j}$ and $x_{i',j'}$ are adjacent if and only if $i \neq i'$. We call this set of vertices the \textit{base vertices}, because we continue to add vertices to $G$. Let $X_i = \{x_{i,j} \mid 1 \leq j \leq k\}$ be the set corresponding to the vertex $v_i$ in $V(H)$. The partition $\copart$ starts as the set $\{X_1, \dots, X_n\}$. We add vertices to these sets, as we continue to construct $G$ in the following three steps: 

\begin{enumerate}
    \item For every $v_iv_j\in E(H)$, add two new vertices $z_{j,i}$ and $z_{i,j}$ to $V(G)$. Add the edges from $z_{j,i}$, $z_{i,j}$ to all base vertices except those in $X_i$ and $X_j$. Assign the vertex $z_{j,i}$ to $X_j$ and the vertex $z_{i,j}$ to $X_i$. Note that any vertex $z_{a,b}$ lives in $X_a$.
    \item For every $v_iv_j\notin E(H)$, add one new vertex $w_{i,j}$ to $V(G)$. Since both $v_i$ and $v_j$ have degree at least $1$ and are not adjacent to each other, there exists at least one other vertex $v_l \in H'$, with a corresponding set $X_l$.  Add the edges from $w_{i,j}$ to all base vertices except those in $X_i$, $X_j$, and $X_l$. Assign the vertex $w_{i,j}$ to $X_l$. 
    \item  If $u_i$ is an isolated vertex in $H$, then the corresponding set in $\copart$ must be a $k$-dominating set with exactly $k$-vertices (so that it does not form a $k$-coalition with any other set). Take the complete graph $K_{mk}$, with vertex set $\{y_{i,1}, \dots, y_{i,k} \mid 1 \leq i \leq m\}$, and graph join $K_{mk}$ and $V(G)$. Then we partition the vertices into sets. For every $1\leq i \leq m$, add the set $Y_i = \{y_{i,1}, \dots, y_{i,k}\}$ to $\copart$ . 
\end{enumerate}

Set $\copart = \{ X_1, \dots, X_n\} \cup \{Y_1, \dots, Y_m\}$. Each set $X_i$ corresponds to the vertex $v_i$ and every set $Y_i$ corresponds to the vertex $u_i$.
We first verify that $\copart$ is a $k$-coalition partition of $G$. Every set $Y_i \in \copart$ is a $k$-dominating set of $G$ with $k$ vertices.
Since $\deg(v_i)\geq 1$, each vertex $v_i$ must be adjacent to at least one vertex $v_j$, where $i\neq j$. 
We check that $X_i$ and $X_j$ form a $k$-coalition. 
The sets $X_i$ and $X_j$ are not $k$-dominating sets because $X_i$ does not $k$-dominate $z_{j,i}$ and $X_j$ does not $k$-dominate $z_{i,j}$. Moreover, all the base vertices in $V(G) \setminus (X_i \cup X_j)$ are adjacent to the base vertices in $X_i$ and $X_j$. For any edge $v_kv_l \in E(H) \setminus \{v_iv_j\}$, the two added vertices $z_{k,l}$ and $z_{l,k}$ are adjacent to the base vertices in $X_i$, $X_j$, or both. Similarly, for any non-edge $v_kv_l$, the added vertex $w_{k,l}$ is adjacent to the base vertices of $X_i$ or $X_j$, or $w_{k,l}$ is contained in one of the two sets. 
Since all vertices in  $G \setminus (X_i \cup X_j)$ are adjacent to the base vertices in $X_i \cup X_j$, we have that $X_i \cup X_j$ $k$-dominates $G$. Therefore, $X_i$ and $X_j$ form a $k$-coalition.

Lastly, we observe that $\kcograph(\copart, G) = H$. Every isolated vertex $w_i \in V(G)$ corresponds to the set $Y_i \in \copart$ and $Y_i$ is a $k$-dominating set of $H$ with $k$ vertices. Therefore, $Y_i$ does not have any $k$-coalition partners in $\copart$. If $v_iv_j \in E(H)$, then, by the argument above, $X_i$ and $X_j$ form a $k$-coalition. If $v_iv_j \notin E(H)$, then there exists a vertex $w_{i,j}$ that is not adjacent to any vertex in $X_i$ or $X_j$. Therefore, $X_i \cup X_j$ does not $k$-dominate $w_{i,j}$ and are not $k$-coalition partners.
\end{proof}

    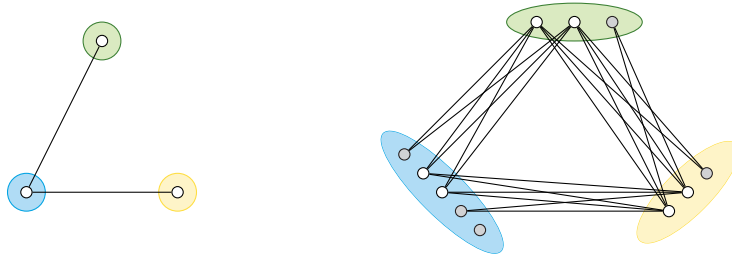
\begin{figure} [htb]
    \centering
        \begin{tikzpicture}[
  every node/.style={circle, draw, fill=white, inner sep=1.5pt}, 
  scale=1
]

\node (A1) at (1,0) {};
\node (A2) at (3,0) {};
\node (A3) at (2,2) {};
\draw  (A1) -- (A2);
\draw  (A1) -- (A3);

\begin{scope} [on background layer]
     \node[ellipse, draw=Cerulean, fill=Cerulean!30, minimum width=0.5cm, minimum height=0.5cm, fit=(A1)] {};
     \node[ellipse, draw=Goldenrod, fill=Goldenrod!30, minimum width=0.5cm, minimum height=0.5cm, fit=(A2)] {};
     \node[ellipse, draw=OliveGreen,  fill=LimeGreen!30, minimum width=0.5cm, minimum height=0.5cm, fit=(A3)] {};
\end{scope}

\node (B1) at (6.25,0.25) {};
\node (B2) at (6.5,0) {};
\node (C1) at (9.5,-0.25) {};
\node (C2) at (9.75,0) {};
\node (D1) at (7.75,2.25) {};
\node (D2) at (8.25,2.25) {};

\draw (B1) -- (D1);
\draw (B1) -- (D2);
\draw (B2) -- (D1);
\draw (B2) -- (D2);
\draw (C1) -- (B1);
\draw (C1) -- (B2);
\draw (C2) -- (B1);
\draw (C2) -- (B2);
\draw (C1) -- (D1);
\draw (C1) -- (D2);
\draw (C2) -- (D1);
\draw (C2) -- (D2);


\node [draw= Black, fill= Gray!40] (B3) at (6,0.5) {};
\node  [draw= Black, fill= Gray!40] (B4) at (6.75,-0.25) {};
\node [draw= Black, fill=Gray!40] (B5) at (7,-0.5) {};
\node [draw= Black, fill= Gray!40](C3) at (10,0.25) {};
\node  [draw=Black, fill= Gray!40]  (D3) at (8.75,2.25) {};




\draw (B3) -- (D1);
\draw(B3) -- (D2);
\draw (C3) -- (D1);
\draw (C3) -- (D2);

\draw (B4) -- (C1);
\draw (B4) -- (C2);
\draw (D3) -- (C1);
\draw (D3) -- (C2);

\begin{scope} [on background layer]
    \draw[Cerulean, rotate around ={-45:(6.5,0)}] (6.5,0) ellipse  (1.1 and 0.3);
    \fill[Cerulean!30, rotate around ={-45:(6.5,0)}] (6.5,0) ellipse (1.1 and 0.3);
    \draw[Goldenrod, rotate around ={45:(9.75,0)}] (9.75,0) ellipse  (0.9 and 0.3);
    \fill[Goldenrod!30, rotate around ={45:(9.75,0)}] (9.75,0) ellipse (0.9 and 0.3);
    \node[ellipse, draw=OliveGreen,  fill=LimeGreen!30, minimum width=0.5cm, minimum height=0.5cm, fit=(D1)(D2)(D3)] {};
    
     %
\end{scope}

\end{tikzpicture}
\caption{(Left) A graph $H$ and (Right) the graph $G$ and $k$-coalition partition $\copart$ such that $\kcograph(\copart,G) = H$ (see \Cref{Theorem: every graph is a k-coalition partition graph}). The base vertices are black with white fill and the non-base vertices are gray. }
\label{fig: theorem construction k-coalition partition graph}
\end{figure}

\section{Acknowledgments}
This research was conducted at the University of Minnesota Duluth REU with support from Jane Street Capital, NSF Grant 2409861, and donations from Ray Sidney and Eric Wepsic. I am grateful to Joe Gallian and Colin Defant for this research opportunity. I am thankful for Noah Kravitz, who suggested many interesting directions and provided helpful comments, as well as Alex Moon for his careful proofreading. I would also like thank Eliot Hodges, Mitchell Lee, Rupert Li, and Maya Sankar for their time and dedication to advising the program.

\bibliographystyle{plain} 
\bibliography{references}

@article {HHHMM20,
    AUTHOR = {Haynes, Teresa W. and Hedetniemi, Jason T. and Hedetniemi,
              Stephen T. and McRae, Alice A. and Mohan, Raghuveer},
     TITLE = {Introduction to coalitions in graphs},
   JOURNAL = {AKCE Int. J. Graphs Comb.},
  FJOURNAL = {AKCE International Journal of Graphs and Combinatorics},
    VOLUME = {17},
      YEAR = {2020},
    NUMBER = {2},
     PAGES = {653--659},
      ISSN = {0972-8600,2543-3474},
   MRCLASS = {05C69},
  MRNUMBER = {4169783},
       DOI = {10.1080/09728600.2020.1832874},
       URL = {https://doi.org/10.1080/09728600.2020.1832874},
}

@article {JAB25,
    AUTHOR = {Jafari, Abbas and Alikhani, Saeid and Bakhshesh, Davood},
     TITLE = {{$k$}-coalitions in graphs},
   JOURNAL = {Australas. J. Combin.},
  FJOURNAL = {The Australasian Journal of Combinatorics},
    VOLUME = {92},
      YEAR = {2025},
     PAGES = {194--209},
      ISSN = {1034-4942,2202-3518},
   MRCLASS = {05C69},
  MRNUMBER = {4926678},
}

@article{BHS25,
  title={On $k$-coalition in graphs: bounds and exact values},
  author={Bre{\v{s}}ar, Bo{\v{s}}tjan and Henning, Michael A. and Samadi, Babak},
  journal={arXiv preprint arXiv:2507.18306},
  year={2025}
}

@article {HHHMMBounds,
    AUTHOR = {Haynes, Teresa W. and Hedetniemi, Jason T. and Hedetniemi,
              Stephen T. and McRae, Alice A. and Mohan, Raghuveer},
     TITLE = {Upper bounds on the coalition number},
   JOURNAL = {Australas. J. Combin.},
  FJOURNAL = {The Australasian Journal of Combinatorics},
    VOLUME = {80},
      YEAR = {2021},
     PAGES = {442--453},
      ISSN = {1034-4942,2202-3518},
   MRCLASS = {05C69},
  MRNUMBER = {4284769},
MRREVIEWER = {Fatiha\ Bendali},
}

@article{TotalKCoalition,
    AUTHOR = {Bre\v{s}ar, Bo\v{s}tjan and Klav\v{z}ar, Sandi and Samadi, Babak},
     TITLE = {Total {$k$}-coalition: bounds, exact values and an application
              to double coalition},
   JOURNAL = {Discrete Math. Theor. Comput. Sci.},
  FJOURNAL = {Discrete Mathematics \& Theoretical Computer Science. DMTCS.},
    VOLUME = {27},
      YEAR = {2025},
    NUMBER = {3},
     PAGES = {Paper No. 3, 18},
      ISSN = {1365-8050},
   MRCLASS = {05C69},
  MRNUMBER = {4952284},
       DOI = {10.46298/dmtcs.15231},
       URL = {https://doi.org/10.46298/dmtcs.15231},
}

@article{EdgeCoalition25,
  title={A Note on Edge Coalitions in Graphs},
  author={Besharati, Nazli and Emadi, Azam Sadat and Masoumi, Iman},
  journal={arXiv preprint arXiv:2507.19871},
  year={2025}
}

@article{FairCoalition25,
    AUTHOR = {Alikhani,Saeid  and Jafari, Abbas and Safazadeh, Maryam},
     TITLE = {Fair coalition in graphs},
   JOURNAL = {J. Iran. Math. Soc.},
  FJOURNAL = {Journal of the Iranian Mathematical Society},
    VOLUME = {6},
      YEAR = {2025},
    NUMBER = {2},
     PAGES = {133--147},
      ISSN = {2717-1612},
   MRCLASS = {05C25 (05C60)},
  MRNUMBER = {4982885},
       DOI = {10.30504/jims.2025.538329.1273},
       URL = {https://doi.org/10.30504/jims.2025.538329.1273},
}

@article {MinMinCoalition,
    AUTHOR = {Bakhshesh, Davood and Henning, Michael A.},
     TITLE = {The minmin coalition number in graphs},
   JOURNAL = {Aequationes Math.},
  FJOURNAL = {Aequationes Mathematicae},
    VOLUME = {99},
      YEAR = {2025},
    NUMBER = {1},
     PAGES = {223--236},
      ISSN = {0001-9054,1420-8903},
   MRCLASS = {05C69 (68R10)},
  MRNUMBER = {4874962},
       DOI = {10.1007/s00010-024-01045-5},
       URL = {https://doi.org/10.1007/s00010-024-01045-5},
}

@article {HHHMM23CCO,
    AUTHOR = {Haynes, Teresa W. and Hedetniemi, Jason T. and Hedetniemi,
              Stephen T. and McRae, Alice A. and Mohan, Raghuveer},
     TITLE = {Coalition graphs},
   JOURNAL = {Commun. Comb. Optim.},
  FJOURNAL = {Communications in Combinatorics and Optimization},
    VOLUME = {8},
      YEAR = {2023},
    NUMBER = {2},
     PAGES = {423--430},
      ISSN = {2538-2128,2538-2136},
   MRCLASS = {05C69},
  MRNUMBER = {4570028},
       DOI = {10.7494/opmath.2023.43.2.173},
       URL = {https://doi.org/10.7494/opmath.2023.43.2.173},
}

@article {HHHMM23DMGT,
    AUTHOR = {Haynes, Teresa W. and Hedetniemi, Jason T. and Hedetniemi,
              Stephen T. and McRae, Alice A. and Mohan, Raghuveer},
     TITLE = {Coalition graphs of paths, cycles, and trees},
   JOURNAL = {Discuss. Math. Graph Theory},
  FJOURNAL = {Discussiones Mathematicae. Graph Theory},
    VOLUME = {43},
      YEAR = {2023},
    NUMBER = {4},
     PAGES = {931--946},
      ISSN = {1234-3099,2083-5892},
   MRCLASS = {05C69},
  MRNUMBER = {4622716},
MRREVIEWER = {Andrei\ Gagarin},
       DOI = {10.7151/dmgt.2416},
       URL = {https://doi.org/10.7151/dmgt.2416},
}

@article {HHHMM23OM,
    AUTHOR = {Haynes, Teresa W. and Hedetniemi, Jason T. and Hedetniemi,
              Stephen T. and McRae, Alice A. and Mohan, Raghuveer},
     TITLE = {Self-coalition graphs},
   JOURNAL = {Opuscula Math.},
  FJOURNAL = {Opuscula Mathematica},
    VOLUME = {43},
      YEAR = {2023},
    NUMBER = {2},
     PAGES = {173--183},
      ISSN = {1232-9274,2300-6919},
   MRCLASS = {05C69},
  MRNUMBER = {4567777},
MRREVIEWER = {Tomislav\ Do\v sli\'c},
       DOI = {10.7494/opmath.2023.43.2.173},
       URL = {https://doi.org/10.7494/opmath.2023.43.2.173},
}

@article{DoubleCoalition2,
    AUTHOR = {Henning, Michael A. and Mojdeh, Doost Ali},
     TITLE = {Double coalitions in regular graphs},
   JOURNAL = {Graphs Combin.},
  FJOURNAL = {Graphs and Combinatorics},
    VOLUME = {41},
      YEAR = {2025},
    NUMBER = {4},
     PAGES = {Paper No. 74, 19},
      ISSN = {0911-0119,1435-5914},
   MRCLASS = {05C69},
  MRNUMBER = {4915176},
       DOI = {10.1007/s00373-025-02937-2},
       URL = {https://doi.org/10.1007/s00373-025-02937-2},
}

@article {DoubleCoalition1,
    AUTHOR = {Henning, Michael A. and Mojdeh, Doost Ali},
     TITLE = {Double coalitions in graphs},
   JOURNAL = {Bull. Malays. Math. Sci. Soc.},
  FJOURNAL = {Bulletin of the Malaysian Mathematical Sciences Society},
    VOLUME = {48},
      YEAR = {2025},
    NUMBER = {2},
     PAGES = {Paper No. 51, 19},
      ISSN = {0126-6705,2180-4206},
   MRCLASS = {05C69},
  MRNUMBER = {4862406},
       DOI = {10.1007/s40840-025-01831-7},
       URL = {https://doi.org/10.1007/s40840-025-01831-7},
}

@article{ConnectedCoalition,
    AUTHOR = {Alikhani, Saeid and Bakhshesh, Davood and Golmohammadi,
              Hamidreza and Konstantinova, Elena V.},
     TITLE = {Connected coalitions in graphs},
   JOURNAL = {Discuss. Math. Graph Theory},
  FJOURNAL = {Discussiones Mathematicae. Graph Theory},
    VOLUME = {44},
      YEAR = {2024},
    NUMBER = {4},
     PAGES = {1551--1566},
      ISSN = {1234-3099,2083-5892},
   MRCLASS = {05C69 (05C85)},
  MRNUMBER = {4788114},
       DOI = {10.7151/dmgt.2509},
       URL = {https://doi.org/10.7151/dmgt.2509},
}

@article{IndependentCoalition,
    AUTHOR = {Alikhani, Saeid and Bakhshesh, Davood and Golmohammadi,
              Hamidreza and Kla\v{v}zar, Sandi},
     TITLE = {On independent coalition in graphs and independent coalition
              graphs},
   JOURNAL = {Discuss. Math. Graph Theory},
  FJOURNAL = {Discussiones Mathematicae. Graph Theory},
    VOLUME = {45},
      YEAR = {2025},
    NUMBER = {2},
     PAGES = {533--544},
      ISSN = {1234-3099,2083-5892},
   MRCLASS = {05C69},
  MRNUMBER = {4895980},
       DOI = {10.7151/dmgt.2543},
       URL = {https://doi.org/10.7151/dmgt.2543},
}

\appendix
\section{Proof of \Cref{Theorem: Multipartite k-coalition number}} \label{sect: complete multipartite graphs}

Before presenting the main argument, we prove several useful lemmas. We first establish a general lower bound on the $k$-coalition number, which is shown to be optimal in certain cases, particularly when $k$ is close to $s_1$.

\begin{lemma} \label{Lemma: R-partite CO_k lower bound}
   Fix $k \geq 2$. For every $K_{s_1, \dots, s_r}$ with $k < s_1$ it follows that $$\COk(K_{s_1, \dots, s_r}) \geq \left(\sum_{i=2}^{r}s_i \right) - k +3.$$
\end{lemma}

\begin{proof}
 It is sufficient to construct a $k$-coalition partition with exactly $(\sum_{i=2}^{r}s_i) - k +3$ sets. For every $1 \leq i \leq r$, label the vertices of $S_i$ as $\{v_{i,1}, v_{i,2}, \dots, v_{i,s_i}\}$. Define $X_1 \subseteq S_1 \cup S_2$ and $X_2 \subseteq S_1$ as follows:
    \begin{align*}
        &X_1 = \{v_{1, 1}, v_{1, 2}, \dots, v_{1, k}\} \cup \{v_{2, 1}, v_{2, 2}, \dots , v_{2, k-1}\}; \\
        &X_2 = \{v_{1, k+1},  \dots, v_{1, s_1}\}.
    \end{align*}
Set
    \begin{equation*}
        \copart = \left\{ X_1, X_2\right\} \cup  \left\{ \{v\} \mid v \in V(K_{s_1, \dots, s_r}) \setminus (X_1 \cup X_2) \right\}.
    \end{equation*}
   The partition $\copart$ is a $k$-coalition partition, where every set is a singleton except $X_1$ and possibly $X_2$. Indeed, no set is individually a $k$-dominating set, but $X_1$ forms a $k$-coalition with every set in $\copart$. See \Cref{fig:Multipartite Lower Bound} for an example. 
    We deduce that 
    \begin{align*}
        \COk(K_{s_1, \dots, s_r}) \geq  |V| - |X_1 \cup X_2| + 2 = \left( \sum_{i=1}^{r}s_i \right) - (s_1 + k - 1) + 2  = \left(\sum_{i=2}^{r}s_i \right) - k +3. \quad \quad \qedhere
    \end{align*}
\end{proof}

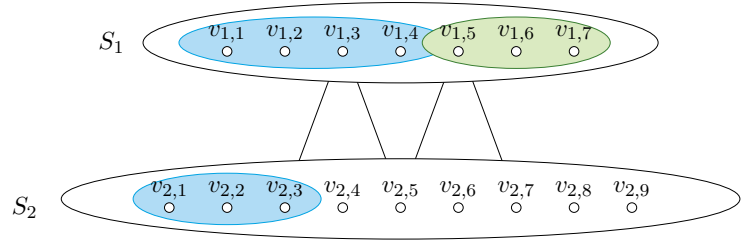
\begin{figure} [htb]
    \centering
    \scalebox{0.85}{
        \begin{tikzpicture}[
  every node/.style={circle, draw, fill=white, inner sep=1.5pt}, 
  scale=0.9
]
  \tikzstyle{c} = [circle, draw, inner sep=1.5pt, minimum size=3pt]

  \foreach \i in {1,2,3,4,5,6,7} {
    \node[c]  at (\i, 3) {};
        \node[draw=none, fill = none] (s1\i)   at (\i, 3.15) {} ;
    \node[draw=none, fill = none] at (\i, 3.3) {$v_{1,\i}$};
    }
    
  \foreach \i in {1,2,3,4,5,6,7,8, 9} {
    \node[c]  at (\i -1 , 0.3) {};
    \node[draw=none, fill = none] (s2\i)   at (\i-1, 0.45) {} ;
    \node[draw=none, fill = none]  at (\i-1, 0.6) {$v_{2,\i}$};
    }
    
\node[draw=none, fill=none] at (-1, 3.15) {\( S_1 \)};
\node[draw=none, fill=none] at (-2.5, 0.3)  {\( S_2 \)};

\begin{scope} [on background layer]
    \draw   (s13) -- (s23);
  \draw   (s13) -- (s25);
  \draw    (s15) -- (s25);
  \draw    (s15) -- (s27);  
\end{scope}

\begin{scope} [on background layer]
     \node[ellipse, draw, fill=white, fit=(s11) (s17), minimum width=3.5cm, minimum height=1.25cm] {};
  \node[ellipse, draw, fill = white, fit=(s21) (s29), minimum width=3.5cm, minimum height=1.25cm] {}; 
\end{scope}

\begin{scope}[on background layer]
\node[ellipse, draw = Cerulean, fill =   Cerulean!30, minimum width=1.4cm, minimum height=0.8cm, fit=(s11)(s14)] {};
\node[ellipse, draw = Cerulean, fill =   Cerulean!30, minimum width=1.2cm, minimum height=0.8cm, fit=(s21)(s23)] {};
\node[ellipse, draw = OliveGreen, fill = LimeGreen!30, minimum width=1.2cm, minimum height=0.8cm, fit=(s15)(s17)] {};
\end{scope}

\end{tikzpicture}
    }
    \caption{The construction given in \Cref{Lemma: R-partite CO_k lower bound} for a $4$-coalition partition of $K_{7,9}$. 
    The set $X_1$ is circled in blue, $X_2$ is circled in green, and all remaining vertices are singleton sets.}
    \label{fig:Multipartite Lower Bound}
\end{figure}

  \begin{lemma} \label{Lemma: Min dominating set in r-partite graphs}
 Fix $k \geq 2$ and $k < s_1$. Let $D$ be a $k$-dominating set of $K_{s_1, \dots, s_r}$. Suppose that $D$ intersects $m$ partite sets and $D$ does not contain any partite set (entirely). 
 Then $$|D| \geq k + \left\lceil \frac{k}{m-1}\right\rceil.$$ 
\end{lemma}

\begin{proof}
    Suppose that $D$ intersects the partite set $S_i$. Since $D$ does not contain $S_i$, there exist vertices in $S_i$ that are not in $D$. To $k$-dominate these vertices, it must be that $|D \cap (V \setminus S_i) | \geq k$. Therefore, it follows that $m \geq 2$. By the pigeonhole principle, there exists at least one partite set $S_j$ such that $|D \cap S_j| \geq \lceil \frac{k}{m-1}\rceil$. Since $D$ also does not contain $S_j$, to $k$-dominate the vertices in $S_j \cap (V \setminus D)$, we have $|D \cap (V \setminus S_j)| \geq k$. Thus, 
    \begin{equation*}
        |D| = |D \cap (V \setminus S_j)| + |D \cap S_j| \geq k + \left\lceil \frac{k}{m-1}\right\rceil.  \qedhere
    \end{equation*}
\end{proof}

\begin{lemma} \label{Lemma: No k-dom sets of size k}
There are no $k$-dominating sets of $K_{s_1, \dots, s_r}$ with exactly $k$ vertices. 
\end{lemma}
\begin{proof}
    Let $D$ be a $k$-dominating set of $K_{s_1, \dots, s_r}$.
    If $D$ contains an entire partite set $S_i$, then $|D| \geq s_i \geq s_1 > k$, as required. We may now assume there exists at least one vertex $v \in S_1 \cap (V \setminus D)$. To $k$-dominate $v$, the set $D$ must include vertices from $V (G) \setminus S_1$, implying that $D$ intersects at least two sets. By \Cref{Lemma: Min dominating set in r-partite graphs}, with $m \geq 2$, 
    \begin{equation*}
        |D| \geq k +  \left\lceil \frac{k}{m-1}\right\rceil \geq k+1.\qedhere
    \end{equation*}
\end{proof}

\begin{lemma} \label{Lemma: There exists a singleton in max size partitions}
    Suppose $r\geq 3$. Let $\copart = \{X_1, X_2, \dots,X_{\COk(K_{s_1,\dots, s_r})}\}$ be a $k$-coalition partition of $K_{s_1,\dots, s_r}$ with the maximum number of sets. Then $\copart$ contains at least one singleton set.
\end{lemma}
  \begin{proof} 
         We proceed by contradiction. 
         Suppose that $|X_i|\geq 2$ for all $X_i \in \copart$. By \Cref{Lemma: No k-dom sets of size k}, $K_{s_1,\dots, s_r}$ has no $k$-dominating sets of size $k$. Therefore, there exist at least two sets $X_1, X_2 \in \copart$ that form a $k$-coalition and $|X_1 \cup X_2| > k$. Since all other sets in $\copart$ consist of at least two vertices, it follows that 
        $$| \copart| < \frac{(\sum_{i=1}^{r} s_i ) - k }{2} + 2 \leq \frac{(\sum_{i=1}^{r} s_i  - k) + (\sum_{i=2}^{r} s_i - s_1  - k)}{2} + 2 = \left( \sum_{i=2}^{r} s_i \right)  - k + 2,$$ where the second inequality follows from the observation that $\sum_{i=2}^{r} s_i  - s_1-k  \geq 0$ when $r \geq 3$. 
        However, by \Cref{Lemma: R-partite CO_k lower bound}, $\COk(K_{s_1, \dots, s_r}) \geq \sum_{i=2}^{r} s_i - k +3$. Therefore, $\copart$ is not a maximum-size $k$-coalition partition, a contradiction.
    \end{proof}

    \begin{lemma} \label{Lemma: Contains a partite set}
        Let $\copart$ be a maximum-size $k$-coalition partition of $K_{s_1,\dots, s_r}$. If $\copart$ has $k$-coalition partners $X_1$ and $X_2$ such that $X_1 \cup X_2$ contains an entire partite set, then $\COk(G) \leq \sum_{i=2}^{r}s_i - k +3$.
    \end{lemma}

    \begin{proof}
        Suppose that $\copart = \{X_1, X_2, \dots,X_{\COk(K_{s_1,\dots, s_r})}\}$ is such a $k$-coalition partition and $S_j \subseteq X_1 \cup X_2$. Note that neither $X_1$ nor $X_2$ can individually contain the entire partite set $S_j$, since then they would then be $k$-dominating. If $\copart = \{X_1,X_2\}$, then $\COk(G) = 2$ and we are done. We may assume there exists another set $X_3 \in \copart$. By \Cref{Observation: No max size partition has k-dom set}, $X_3$ is not a $k$-dominating set, and therefore has a $k$-coalition partner, $X_4$. By \Cref{Lemma: No k-dom sets of size k}, we have $|X_3 \cup X_4| \geq k+1$. 
        If $X_4 \notin \{X_1, X_2\}$, then
        \begin{equation*}
            |\copart| \leq \sum_{i=1}^r s_i - |X_1 \cup X_2| - |X_3\cup X_4| + 4 \leq  \sum_{i=1}^r s_i - s_j - (k+1) + 4  \leq  \sum_{i=2}^r s_i -k +3,
        \end{equation*}
        where the last inequality follows from the fact that $s_j \geq s_1$.  
        If $X_4 \in \{X_1, X_2\}$, suppose without loss of generality that $X_4 = X_2$. In order to $k$-dominate the vertices in $X_1 \cap S_j$, we have $| (X_2 \cup X_3) \cap (V(K_{s_1,\dots,s_r}) \setminus S_j)| \geq k$. Hence
 \begin{align*}
            |\copart| &\leq \sum_{i=1}^r s_i - |X_1 \cup X_2\cup X_3| + 3
            \leq \sum_{i=1}^r s_i - s_j - k + 3  \leq  \sum_{i=2}^r s_i -k +3. \qedhere
        \end{align*}
    \end{proof}


    \begin{lemma} \label{Lemma: m not zero}
        Suppose $k \geq r$ and $r >1$, then $k + \big\lceil\frac{k}{r-1}\big\rceil -1 \not\equiv 0 \pmod{r}$.  
    \end{lemma}

    \begin{proof}
        Write $k = q(r-1) + d$ for some $q\in \ints_{\geq 0}$ and $0\leq d <r-1$. First, suppose $d = 0$. Then 
        $$ k + \left\lceil\frac{k}{r-1}\right\rceil -1  = q(r-1) + q-1 = qr -1 \not\equiv 0 \pmod{r}.$$
        Next, suppose $1\leq d < r-1$. Then $\big\lceil\frac{k}{r-1}\big\rceil = \big\lceil q + \frac{d}{r-1}\big\rceil = q + 1$, and so 
        \begin{equation*}
            k + \left\lceil\frac{k}{r-1}\right\rceil -1  = (q(r-1) + d) +(q+1) -1 = qr +d \not\equiv 0 \pmod{r}. \qedhere
        \end{equation*}
    \end{proof}

We now have the tools required to prove the main theorem. 
\begin{proof} [Proof of \Cref{Theorem: Multipartite k-coalition number}] \setcounter{tbox}{0}
We proceed by casework.



We may now assume $k\geq 2$. Let $\copart = \{X_1, X_2, \dots,X_{\COk(K_{s_1,\dots, s_r})}\}$ be a maximum-size $k$-coalition partition of $\COk(K_{s_1,\dots, s_r})$. By \Cref{Lemma: There exists a singleton in max size partitions}, there is at least one singleton set $X_1 \in \copart$, which must have at least one coalition partner $X_2 \in \copart$.
By \Cref{Lemma: No k-dom sets of size k}, $|X_1 \cup X_2| \geq k +1$, and consequently $|X_2| \geq k$.

\sta{\label{Multipartite Case 2} If $X_2$ forms a $k$-coalition with every set in $\copart$, then $|\copart| \leq \sum_{i=2}^r s_i - k + 3$.}

By \Cref{Lemma: max number of coalition partners}, the set $X_2$ has at most $\Delta(G) - k +2$ distinct $k$-coalition partners. Since $\Delta(K_{s_1,\dots, s_r}) = \sum_{i=2}^r s_i$, it follows that $|\copart| \leq 1 + \Delta(G) - k + 2 = \sum_{i=2}^{r} s_i -k +3$, which proves \eqref{Multipartite Case 2}.
\vspace{3mm}

Now suppose there exists some set $X_3 \in \copart$ that does not form a $k$-coalition with $X_2$. By \Cref{Lemma: No k-dom sets of size k}, the set $X_3$ cannot be a $k$-dominating set itself, and so it has a $k$-coalition partner $X_t$, where $t \notin \{2,3\}$. In particular, it is possible that $t=1$.
By \Cref{Lemma: Contains a partite set}, if $X_1 \cup X_2$ or $X_3 \cup X_t$ contains an entire partite set, then $|\copart| \leq \sum_{i=2}^r s_i - k + 3$.

Since we lower bound the $k$-coalition number with the construction in \Cref{Lemma: R-partite CO_k lower bound}, if the maximum-size $k$-coalition partition has a structure as described above, then the $k$-coalition number is equal to $\sum_{i=2}^{r} s_i -k +3$. However, these cases are only optimal when $s_1$ is close in size to $k$.

We now examine the cases where the maximum-size $k$-coalition partition does not admit this structure and the $k$-coalition number will exceed this lower bound.
Label the vertices of the partite set $S_i$ as $v_{i,1}, v_{i,2}, \dots, v_{i,s_i}$. Note that $|X_2| \geq k$ and $|X_3 \cup X_t| \geq k+1$.
We now suppose that neither $X_1 \cup X_2$ nor $X_3 \cup X_t$ contains a partite set and consider the following cases:

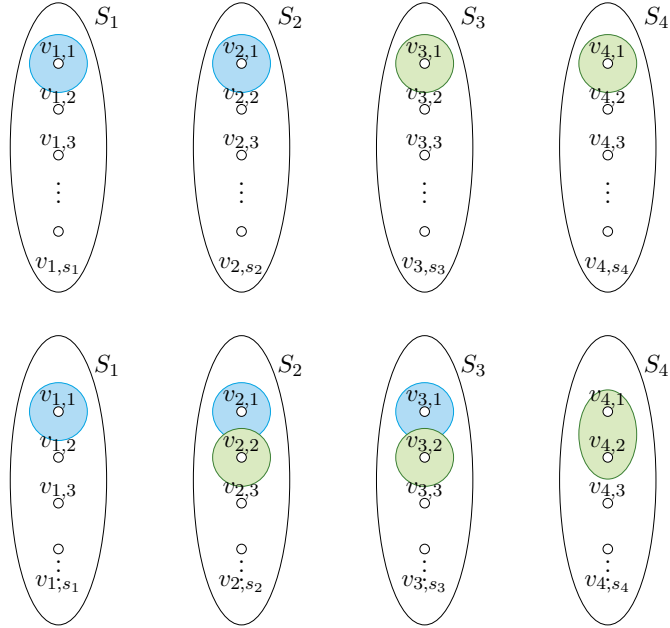
\begin{figure} [hb]
    \centering
    \scalebox{0.85}{
\begin{tikzpicture}[
  every node/.style={circle, draw, fill=white, inner sep=1.5pt},
  scale=0.95
]
\begin{scope} 
[yshift=0.2cm]
  \foreach \i in {1,2,3} {
    \node  (s1\i) at (-3, 6.5 - 0.75*\i) {};
    \node[draw=none, fill=none]  at (-3, 6.7 - 0.75*\i) {$v_{1,\i}$};
    \node (s2\i) at (0, 6.5 - 0.75*\i) {};
    \node[draw=none, fill=none]  at (0, 6.7 - 0.75*\i) {$v_{2,\i}$};
    \node (s3\i) at (3, 6.5 - 0.75*\i) {};
    \node[draw=none, fill=none] at (3, 6.7 - 0.75*\i) {$v_{3, \i}$};
    \node (s4\i) at (6, 6.5 - 0.75*\i) {};
    \node[draw=none, fill=none] at (6, 6.7 - 0.75*\i) {$v_{4, \i}$};
  }

  \node[draw=none, fill=none] at (-3, 3.75) {$\vdots$};
  \node[draw=none, fill=none] at (0, 3.75)  {$\vdots$};
  \node[draw=none, fill=none] at (3, 3.75)  {$\vdots$};
  \node[draw=none, fill=none] at (6, 3.75)  {$\vdots$};

  \node  [label=below:$v_{1,s_1}$] (s14)  at (-3, 3) {};
  \node  [label=below:$v_{2,s_2}$] (s24) at (0, 3) {};
  \node  [label=below:$v_{3,s_3}$] (s34) at (3, 3) {};
  \node  [label=below:$v_{4,s_4}$] (s44) at (6, 3) {};

  \begin{scope}[on background layer]
    \node[ellipse, draw, minimum width=1.5cm, minimum height=4.5cm, fit=(s11)(s12)(s13)(s14)] {};
    \node[ellipse, draw, minimum width=1.5cm, minimum height=4.5cm, fit=(s21)(s22)(s23)(s24)] {};
    \node[ellipse, draw, minimum width=1.5cm, minimum height=4.5cm, fit=(s31)(s32)(s33)(s34)] {};
    \node[ellipse, draw, minimum width=1.5cm, minimum height=4.5cm, fit=(s41)(s42)(s43)(s44)] {};
  \end{scope}

  \begin{scope}[on background layer]
    \node[ellipse, draw=Cerulean, fill=  Cerulean!30, minimum width=0.9cm, minimum height=0.9cm, fit=(s11)] {};
    \node[ellipse, draw=Cerulean, fill=  Cerulean!30, minimum width=0.9cm, minimum height=0.9cm, fit=(s21)] {};
  \end{scope}

  \begin{scope}[on background layer]
    \node[ellipse, draw=OliveGreen, fill=LimeGreen!30, minimum width=0.9cm, minimum height=0.9cm, fit=(s31) ] {};
   \node[ellipse, draw=OliveGreen, fill=LimeGreen!30, minimum width=0.9cm, minimum height=0.9cm, fit=(s41)] {};
  \end{scope}

  \node[draw=none, fill=none] at (-2.2, 6.5) {\( S_1 \)};
  \node[draw=none, fill=none] at (0.8, 6.5)  {\( S_2 \)};
  \node[draw=none, fill=none] at (3.8, 6.5)  {\( S_3 \)};
  \node[draw=none, fill=none] at (6.8, 6.5) {\( S_4 \)};
\end{scope}


  \begin{scope}
  \foreach \i in {1,2,3} {
    \node (c2s1\i) at (-3, 1 - 0.75*\i) {};
    \node[draw=none, fill=none] at (-3, 1.2 - 0.75*\i) {$v_{1,\i}$};
    \node (c2s2\i) at (0, 1 - 0.75*\i) {};
    \node[draw=none, fill=none] at (0, 1.2 - 0.75*\i) {$v_{2,\i}$};
    \node (c2s3\i) at (3, 1 - 0.75*\i) {};
    \node[draw=none, fill=none] at (3, 1.2 - 0.75*\i) {$v_{3, \i}$};
    \node (c2s4\i) at (6, 1 - 0.75*\i) {};
    \node[draw=none, fill=none] at (6, 1.2 - 0.75*\i) {$v_{4, \i}$};
  }

  \node[draw=none, fill=none] at (-3, -2.25) {$\vdots$};
  \node[draw=none, fill=none] at (0, -2.25)  {$\vdots$};
  \node[draw=none, fill=none] at (3, -2.25)  {$\vdots$};
  \node[draw=none, fill=none] at (6, -2.25) {$\vdots$};

  \node [label=below:$v_{1,s_1}$] (c2s14) at (-3, -2) {};
  \node [label=below:$v_{2,s_2}$] (c2s24) at (0, -2) {};
  \node [label=below:$v_{3,s_3}$] (c2s34) at (3, -2) {};
  \node [label=below:$v_{4,s_4}$] (c2s44) at (6, -2) {};

  \begin{scope}[on background layer]
    \node[ellipse, draw, minimum width=1.5cm, minimum height=4.5cm, fit=(c2s11)(c2s12)(c2s13)(c2s14)] {};
    \node[ellipse, draw, minimum width=1.5cm, minimum height=4.5cm, fit=(c2s21)(c2s22)(c2s23)(c2s24)] {};
    \node[ellipse, draw, minimum width=1.5cm, minimum height=4.5cm, fit=(c2s31)(c2s32)(c2s33)(c2s34)] {};
    \node[ellipse, draw, minimum width=1.5cm, minimum height=4.5cm, fit=(c2s41)(c2s42)(c2s43)(c2s44)] {};
  \end{scope}

  \begin{scope}[on background layer]
    \node[ellipse, draw=Cerulean, fill=  Cerulean!30, minimum width=0.9cm, minimum height=0.9cm, fit=(c2s11)] {};
    \node[ellipse, draw=Cerulean, fill=  Cerulean!30, minimum width=0.9cm, minimum height=0.9cm, fit=(c2s21)] {};
     \node[ellipse, draw=Cerulean, fill=  Cerulean!30, minimum width=0.9cm, minimum height=0.9cm, fit=(c2s31)] {};
  \end{scope}

  \begin{scope}[on background layer]
   \node[ellipse, draw=OliveGreen, fill=LimeGreen!30, minimum width=0.9cm, minimum height=0.9cm, fit=(c2s22) ] {};
    \node[ellipse, draw=OliveGreen, fill=LimeGreen!30, minimum width=0.9cm, minimum height=0.9cm, fit=(c2s32) ] {};
    \node[ellipse, draw=OliveGreen, fill=LimeGreen!30, minimum width=0.9cm, minimum height=0.9cm, fit=(c2s41)(c2s42)] {};
  \end{scope}

  \node[draw=none, fill=none] at (-2.2, 1) {\( S_1 \)};
  \node[draw=none, fill=none] at (0.8, 1)  {\( S_2 \)};
  \node[draw=none, fill=none] at (3.8, 1)  {\( S_3 \)};
  \node[draw=none, fill=none] at (6.8, 1) {\( S_4 \)};
\end{scope}

\end{tikzpicture} }
\caption{The complete $4$-partite graph $K_{s_1,s_2,s_3,s_4}$ with the lower bound constructions for \eqref{Multipartite Case 3} with $k=2$ and $r=4$ (top)  and \eqref{Multipartite Case 4} with $k=3$ and $r=4$ (bottom). The set $Y_1$ is in blue, $Y_2$ is in green, and all other vertices are singletons.} 
    \label{fig: Multipartite Case 3 and Case 4}
\end{figure}

\sta{\label{Multipartite Case 3} If $r \geq 2k$, then $\COk(K_{s_1,\dots,s_r}) =\sum_{i=1}^r s_i - 2k + 2$.}
         \textbf{Upper Bound.} Note that $|V \setminus (X_2 \cup X_3 \cup X_t)| \leq \sum_{i=1}^r s_i - k - (k+1)$. We may upper bound the number of sets in $\copart$ by assuming that all other sets are singletons. Hence
        \begin{equation*}
            \COk(K_{s_1,\dots,s_r}) = |\copart| \leq \sum_{i=1}^r s_i - 2k-1  + 3 = \sum_{i=1}^r s_i - 2k + 2.
        \end{equation*}
        \textbf{Lower Bound.} Let $Y_1 = \{v_{1,1},v_{2,1}, \dots, v_{k,1}\}$ and $Y_2 = \{v_{k+1,1},v_{k+2,1}, \dots, v_{2k,1}\}$. Then the following is a $k$-coalition partition of $K_{s_1, \dots, s_r}$:
        \begin{equation*}
             \copart = \left\{ Y_1, Y_2 \right\} \cup \left\{ \{v\} \mid v \in V \setminus (Y_1 \cup Y_2) \right\}.
        \end{equation*}
       See \Cref{fig: Multipartite Case 3 and Case 4} (left) for an example. The set $Y_1$ contains one vertex from each of $S_1, \dots, S_k$ and the set $Y_2$ contains one vertex from each of $S_{k+1}, \dots, S_{2k}$. We take the remaining vertices to be singleton sets. Then $Y_1$ forms a $k$-coalition with every singleton $\{v\}$ where $v \in V \setminus (S_1 \cup \dots \cup S_k)$ and $Y_2$ forms a $k$-coalition with every singleton $\{v\}$ where $v \in V \setminus (S_{k+1} \cup \dots \cup S_{2k})$. Consequently, 
        $\COk(K_{s_1, \dots, s_r}) \geq \sum_{i=1}^r s_i - 2k + 2,$
        which proves \eqref{Multipartite Case 3}.
       
       \vspace{3mm}

\sta{\label{Multipartite Case 4} If $k < r < 2k $, then $\COk(K_{s_1,\dots,s_r}) =\sum_{i=1}^r s_i - 2k + 1$.}

    \noindent \textbf{Upper Bound.} First, suppose that $|X_1 \cup X_2| > k+1$ (implying that $|X_2| > k  $) or $|X_3 \cup X_t| > k+1$. 
    Then, even if all the other sets in $\copart$ are singleton sets, it follows that $$|\copart| \leq |V \setminus (X_2 \cup X_3 \cup X_t) | + 3 < \sum_{i=1}^r s_i - 2k - 1 +3 = \sum_{i=1}^r s_i - 2k + 2.$$  
       Now assume that $|X_1 \cup X_2| = k+1$ (so, in particular, $|X_2| = k$) and $|X_3 \cup X_t| = k+1$. 
       For contradiction, suppose that all sets in $\copart \setminus \{X_2, X_3, X_t\}$ are singletons. We will argue that this is not possible.
       
        Note that every partite set has at least $3$ vertices, and by \Cref{Lemma: No k-dom sets of size k} every $k$-dominating set has at least $3$ vertices. 
       Therefore no two singleton sets can form a $k$-coalition. By assumption, neither $X_1 \cup X_2$ nor $X_3 \cup X_t$ contains an entire partite set, so \Cref{Lemma: Min dominating set in r-partite graphs} implies that the $k$-dominating sets $X_1 \cup X_2$  and $X_3 \cup X_t$ must each intersect $k+1$ partite sets, with one vertex from each. 
      
       

 If $X_2$ intersects the partite set $S_j$, then there exists at least one vertex $v\in S_j \setminus (X_2 \cup X_3 \cup X_t)$. By assumption, $\{v\}$ is a singleton set in $\copart$. However, $X_2 \cup \{v\}$ is not a $k$-dominating set because it does not $k$-dominate the remaining vertices in $S_j \setminus (X_2 \cup \{v\})$.  If both $|X_3|< k$ and $|X_t| < k$, then neither $X_3$ nor $X_t$ has enough vertices to form a $k$-coalition with $\{v\}$. Hence, $\{v\}$ does not have a $k$-coalition partner, a contradiction.
       
    Hence, either $X_3$ or $X_t$ has exactly $k$ vertices. Without loss of generality, say $|X_3| = k$ and $|X_t| = 1$.
    Then, since $X_2$ and $X_3$ each intersect $k$ distinct partite sets and there are at most $2k-1$ partite sets, there is some partite set $S_l$ intersecting both $X_2$ and $X_3$. Since $S_l$ has at least three vertices, there is at least one vertex $u \in S_l \setminus (X_2 \cup X_3)$. However, $\{u\}$ does not form a $k$-coalition with $X_2$ or $X_3$, since $\{u\} \cup X_3$ does not $k$-dominate the vertices in $X_2\cap S_l$. Similarly  $\{u\} \cup X_2$ does not $k$-dominate the vertices in $X_3\cap S_l$. Therefore, $\{u\}$ does not have a $k$-coalition partner, a contradiction.

    Since there is at least one set in $\copart \setminus \{X_2, X_3,X_t\}$ that is not a singleton set, we have the bound  
       \begin{equation*}
           |\copart| < |V \setminus (X_2 \cup X_3 \cup X_t)| + 3  \leq \sum_{i=1}^r s_i - 2k - 1 + 3 = \sum_{i=1}^r s_i - 2k + 2. 
        \end{equation*} 

       \noindent \textbf{Lower Bound.} If $s_1 \leq k + 2$, then $\sum_{i=1}^r s_i - 2k + 1 \leq  \sum_{i=2}^r s_i - k + 3$ and the construction in \Cref{Lemma: R-partite CO_k lower bound} achieves the maximum. We assume here that $s_1 \geq k + 2$.
       Let $Y_1 = \{v_{1,1},v_{2,1}, \dots, v_{k,1}\}$ and $Y_2 = \{v_{2,2}, \dots, v_{k+1,2}\} \cup \{v_{k+1,1}\}$.  
       Then define
       \begin{equation*}
           \copart = \left\{ Y_1, Y_2 \right\} \cup \left\{ \{v\} \mid v \in V \setminus (Y_1 \cup Y_2) \right\}.
       \end{equation*}
       See \Cref{fig: Multipartite Case 3 and Case 4} (right) for this construction. The set $Y_1$ forms a $k$-coalition with every singleton in $S_i$ where $i \geq k+1$ and $Y_2$ forms a $k$-coalition with every singleton in $S_i$ when $i \neq k+1$. Hence, $\copart$ is a $k$-coalition partition. Therefore, 
       \begin{equation*}
           |\copart| \geq \sum_{i=1}^r s_i - k - (k+1) + 2 = \sum_{i=1}^r s_i - 2k + 1.
       \end{equation*}

\sta{\label{Multipartite Case 5} If $r \leq k$ and $k + \big\lceil \frac{k}{r-1}\big\rceil -1  \pmod r \leq \frac{r}{2}$,  then $$\COk(K_{s_1,\dots,s_r}) = \sum_{i=1}^r s_i - 2k - 2 \left\lceil \frac{k}{r-1}\right\rceil + 4.$$}

\noindent \textbf{Upper Bound.} Since the $k$-dominating sets $X_1 \cup X_2$ and $X_3 \cup X_t$ each intersect at most $k$ partite sets,  \Cref{Lemma: Min dominating set in r-partite graphs} implies that $|X_1 \cup X_2| \geq k + \big\lceil \frac{k}{r-1}\big\rceil$ and $|X_3 \cup X_t| \geq k + \big\lceil \frac{k}{r-1} \big\rceil$.
      By assuming all other sets in the partition are singleton sets, we obtain the bound
       \begin{equation*}
           |\copart| \geq |V \setminus (X_2 \cup X_3 \cup X_t) | + 3  = \sum_{i=1}^r s_i - 2k - 2\left\lceil \frac{k}{r-1}\right\rceil + 4.
       \end{equation*}

         \noindent \textbf{Lower Bound.}  
        When $s_1 \leq k + 2 \big\lceil \frac{k}{r-1} \big\rceil-1$, we have $$\sum_{i=1}^r s_i - 2k - 2\left\lceil \frac{k}{r-1} \right\rceil +4 \leq \sum_{i=2}^r s_i - k + 3.$$ In this case, the construction in \Cref{Lemma: R-partite CO_k lower bound} is a maximum-size $k$-coalition partition. 
        In what follows, we assume that $s_1 \geq k + 2 \left\lceil \frac{k}{r-1} \right\rceil-1$.
         Set $m = k + \lceil \frac{k}{r-1}\rceil -1 \pmod r $ and $n = \lceil \frac{k}{r-1}\rceil$. By \Cref{Lemma: m not zero}, $0<m \leq \frac{r}{2}$ and, by assumption, $s_i \geq 2n-1$ for all $i$. Define
        \begin{align*}
            &Y_1 =    \{v_{i,1}, \dots, v_{i,n} \mid 1 \leq i \leq m \} \cup \{v_{i,1}, \dots, v_{i,n-1} \mid m+1  \leq i \leq r \};\\
            &Y_2 = \{v_{i,n+1}, \dots, v_{i,2n-1} \mid 1 \leq i \leq r-m \} \cup \{v_{i,n}, \dots, v_{i,2n-1} \mid r-m+1  \leq i \leq r \}.
        \end{align*}
        Then set
        \begin{equation*}
            \copart = \left\{ Y_1, Y_2 \right\} \cup \left\{ \{v\} \mid v \in V \setminus (Y_1 \cup Y_2)\right\}.
        \end{equation*}
        See \Cref{fig: Multipartite Case 3} for this construction. 
        Then $Y_1$ and $Y_2$ are not $k$-dominating sets because $Y_1$ does not $k$-dominate the vertices in $(S_1 \cup \dots \cup S_{m}) \setminus Y_1$  and $Y_2$ does not $k$-dominate the vertices in $(S_{r-m+1}\cup \dots \cup S_r)\setminus Y_2$. 
        Moreover, $Y_1$ forms a $k$-coalition with every singleton in $S_{m+1}, S_{m+2}, \dots, S_{r}$ (where  $m + 1 \leq \frac{r}{2} + 1$) and $Y_2$ forms a $k$-coalition with every singleton in $S_{1}, S_{2},\dots, S_{r-m}$ (where $r-m  \geq \frac{r}{2})$. Hence, $\copart$ is a $k$-coalition partition and we have 
        \begin{align*}
             |\copart| \geq |V\setminus(Y_1 \cup Y_2)| + 2 
              \geq  \sum_{i=1}^r s_i - 2k - 2 \left\lceil \frac{k}{r-1} \right\rceil  +4,
         \end{align*}
         which proves \eqref{Multipartite Case 5}.   Note that complete bipartite graphs fall under \eqref{Multipartite Case 5} and our results agree with \Cref{Theorem: Complete bipartite graphs} \cite{BHS25}.

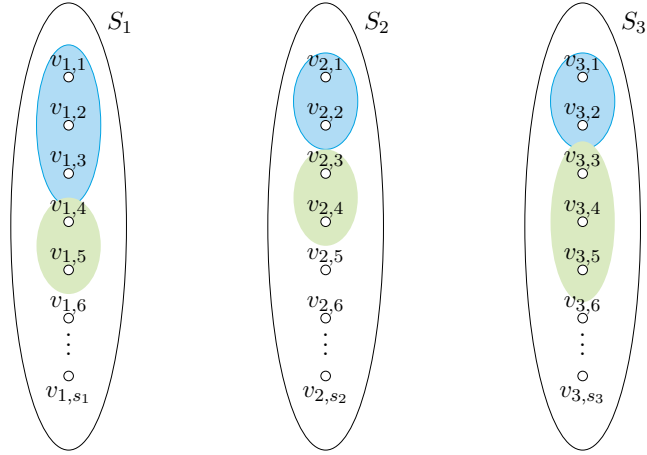
\begin{figure} 
    \centering
    \scalebox{0.85}{
\begin{tikzpicture}[
  every node/.style={circle, draw, fill=white, inner sep=1.5pt}, 
  scale=1
]

\foreach \i in {1,2,3,4,5,6} {
  \node  (s1\i) at (-4, 3 - 0.75*\i) {};
  \node[draw=none, fill = none] at (-4, 3.2 - 0.75*\i) {$v_{1,\i}$};
  \node (s2\i) at (0, 3 - 0.75*\i) {};
    \node[draw=none, fill = none] at (0, 3.2 - 0.75*\i) {$v_{2,\i}$};
  \node (s3\i) at (4, 3 - 0.75*\i) {};
    \node[draw=none, fill = none] at (4, 3.2 - 0.75*\i) {$v_{3, \i}$};
}

\node[draw=none, fill=none] at (-4, -1.8) {$\vdots$};
\node[draw=none, fill=none] at (0, -1.8)  {$\vdots$};
\node[draw=none, fill=none] at (4, -1.8)  {$\vdots$};

\node[draw=none, fill = none] at (-4, -2.7) {$v_{1,s_1}$};
\node[draw=none, fill = none] at (0, -2.7) {$v_{2,s_2}$};
\node[draw=none, fill = none] at (4, -2.7){$v_{3,s_3}$};
\node (s17)  at (-4, -2.4) {};
\node (s27)  at (0, -2.4) {};
\node (s37) at (4, -2.4) {};

\begin{scope}[on background layer]
  \node[ellipse, draw, minimum width=1.8cm, minimum height=4.5cm, fit=(s11)(s12)(s13)(s14)(s15)(s16)(s17)] {};
  \node[ellipse, draw, minimum width=1.8cm, minimum height=4.5cm, fit=(s21)(s22)(s23)(s24)(s25)(s26)(s27)] {};
  \node[ellipse, draw, minimum width=1.8cm, minimum height=4.5cm, fit=(s31)(s32)(s33)(s34)(s35)(s36)(s37)] {};
\end{scope}

\begin{scope}[on background layer]
\node[ellipse, draw = Cerulean, fill =   Cerulean!30, minimum width=1cm, minimum height=2cm, fit=(s11)(s12) (s13)] {};
\node[ellipse, draw = Cerulean, fill =   Cerulean!30, minimum width=1cm, minimum height=1.5cm, fit=(s21)(s22)] {};
\node[ellipse, draw =  Cerulean, fill =   Cerulean!30, minimum width=1cm, minimum height=1.5cm, fit=(s31)(s32)] {};
\end{scope}

\begin{scope}[on background layer]
\node[ellipse, draw = none, fill = LimeGreen!30, minimum width=1cm, minimum height=1.5cm, fit=(s14)(s15) ] {};
\node[ellipse, draw = none, fill = LimeGreen!30, minimum width=1cm, minimum height=1.5cm, fit=(s23)(s24)] {};
\node[ellipse, draw = none, fill = LimeGreen!30, minimum width=1cm, minimum height=2cm, fit=(s33)(s34) (s35)] {};
\end{scope}

\node[draw=none, fill=none] at (-3.2, 3.1) {\( S_1 \)};
\node[draw=none, fill=none] at (0.8, 3.1)  {\( S_2 \)};
\node[draw=none, fill=none] at (4.8, 3.1)  {\( S_3 \)};

\end{tikzpicture} }
    \caption{The complete $3$-partite graph $K_{s_1, s_2, s_3}$ with the $5$-coalition partition attaining the lower bound in \eqref{Multipartite Case 5}. The set $Y_1$ is in blue, the set $Y_2$ is in green, and all other vertices are singleton sets.  }
    \label{fig: Multipartite Case 3}
\end{figure}

\sta{\label{Multipartite Case 6} If $r \leq k$ and $k + \lceil \frac{k}{r-1}\rceil -1 \pmod r > \frac{r}{2}$, then 
$$\COk(K_{s_1,\dots,s_r} )=  \sum_{i=1}^r s_i - 2k - 2\left\lceil \frac{k}{r-1} \right\rceil  + 3.$$}
\noindent \textbf{Upper Bound.} 
For every $r \geq 3$, when $s_1 \leq k + 2 \lceil \frac{k}{r-1} \rceil$, we have $$\sum_{i=1}^r s_i - 2k - 2\left\lceil \frac{k}{r-1} \right\rceil \leq \sum_{i=2}^r s_i - k + 3.$$ In this case, the construction from \Cref{Lemma: R-partite CO_k lower bound} is a maximum-size $k$-coalition partition. We may now assume $s_1 \geq k + 2 \big\lceil \frac{k}{r-1} \big\rceil$. 
As above, $X_1 \cup X_2$ and $ X_3 \cup X_t$ each intersect at most $r$ partite sets (and do not contain an entire partite set), so \Cref{Lemma: Min dominating set in r-partite graphs} implies both $|X_1 \cup X_2| \geq k + \lceil \frac{k}{r-1}\rceil$ and $|X_3 \cup X_t| \geq k + \lceil \frac{k}{r-1}\rceil$. If $|X_1 \cup X_2| > k + \lceil \frac{k}{r-1}\rceil$ or  $|X_3 \cup X_t| > k + \lceil \frac{k}{r-1}\rceil$, then 
        \begin{equation*}
            |\copart| \leq |V \setminus (X_2 \cup X_3\cup X_t)| +3 <  \sum_{i=1}^r s_i - 2k - 2 \left\lceil \frac{k}{r-1} \right\rceil + 4.
        \end{equation*}
       We now assume $|X_1 \cup X_2| = k + \lceil \frac{k}{r-1}\rceil$ 
        and $|X_3 \cup X_t| = k + \lceil \frac{k}{r-1}\rceil$. For contradiction, suppose that all the remaining sets in  $\copart \setminus \{X_1, X_2, X_3, X_t\}$ are singletons. Since no pair of singletons can form a $k$-coalition, we obtain the contradiction by finding at least one singleton set that cannot form a $k$-coalition with $X_2, X_3,$ or $X_t$.
        
        By \Cref{Lemma: Min dominating set in r-partite graphs}, the $k$-dominating sets $X_1 \cup X_2$ and $X_3 \cup X_t$ must each intersect all $r$ partite sets, with at most $\big\lceil \frac{k}{r-1}\big\rceil$ vertices per partite set. Since $s_j \geq s_1 \geq k + 2\lceil\frac{k}{r-1}\rceil$, we have $|S_j \setminus \{X_2, X_3, X_t\}| \geq k $ and so every partite set $S_j$ has at least two vertices $u,v$ that are not in $X_1, X_2, X_3$ or $X_t$.

        Let $m =  k + \lceil \frac{k}{r-1}\rceil -1 \pmod r$. By a counting argument, there are $m$ partite sets $S_j$ such that $|X_2 \cap S_j| = \lceil \frac{k}{r-1}\rceil$ vertices.  Suppose that $|X_3| \geq |X_t|$ (the case where $|X_3| \leq |X_t|$ follows by a symmetric argument).  If $|X_3| < k +\lceil \frac{k}{r-1} \rceil -1$, then $X_t$ and $X_3$ do not have enough vertices to form a $k$-coalition with a singleton. If $|X_3| = k +\lceil \frac{k}{r-1} \rceil -1$ (so $|X_t| = 1$), then there are $m$ distinct partite sets $S_j$ such that $|X_3 \cap S_j| = \lceil \frac{k}{r-1} \rceil$.
        Since $m > \frac{r}{2}$, there is at least one partite set $S_j$ such that both $|X_2 \cap S_j| = \lceil \frac{k}{r-1} \rceil$ and $|X_3 \cap S_j| = \lceil \frac{k}{r-1} \rceil$.
    Then 
    $$|X_2 \cap (V \setminus S_j)| \leq |X_2| - |X_2 \cap S_j| \leq k-1 \quad \text{and} \quad |X_3 \cap (V \setminus S_j)| \leq |X_3| - |X_3 \cap S_j| \leq k-1.$$ 
    Consequently, $X_2$ and $X_3$ do not $k$-dominate $u$ and $v$. Since $u$ and $v$ are non-adjacent, neither $X_2 \cup \{v\}$ nor  $X_3 \cup \{v\}$ can $k$-dominate $u$. Therefore, the singleton set $\{v\}$ does not form a $k$-coalition with $X_2$, $X_3$, or $X_t$.
    We conclude that there is at least one set in $\copart \setminus \{X_2, X_3,X_t\}$ that is not a singleton set. Therefore
         \begin{equation*}
        |\copart| < |V \setminus ( X_2 \cup X_3 \cup X_t)| + 3 =  \sum_{i=1}^r s_i - 2k - 2 \left\lceil \frac{k}{r-1}\right\rceil  + 4.
    \end{equation*}
        \noindent \textbf{Lower Bound.}
        Let $n = \lceil \frac{k}{r-1}\rceil$ and recall that $m> \frac{r}{2}$. Set 
        \begin{align*}
            &Y_1 =   \{v_{i,1}, \dots, v_{i,n} \mid 1 \leq i \leq m \} \cup \{v_{i,1}, \dots, v_{i,n-1} \mid m+1  \leq i \leq r \}; \\
            &Y_2 = \{v_{i,n+1}, \dots, v_{i,2n-1} \mid 1 \leq i \leq r-m \} \cup \{v_{i,n+1}, \dots, v_{i,2n} \mid r-m+1  \leq i \leq r \} \cup \{v_{r, n}\}.
        \end{align*}
        Define
        \begin{equation*}
            \copart = \left\{ Y_1, Y_2 \right\} \cup \left\{ \{v\} \mid v \in V \setminus (Y_1 \cup Y_2)\right\}.
        \end{equation*}
        Note that $Y_1$ and $Y_2$ are not $k$-dominating sets because $Y_1$ does not $k$-dominate any vertex in $(S_1 \cup  \dots \cup S_m) \setminus Y_1$ and $Y_2$ does not $k$-dominate any vertex in $S_r \setminus Y_2$. 
        Since $1<m < r$, the set $Y_1$ forms a $k$-coalition with each singleton $\{v\}$ in $S_{r}$, and the set $Y_2$ forms a $k$-coalition with each singleton $\{v\}$ in $S_{1} \cup S_{2} \cup \dots \cup  S_{r-1}$. Every singleton in $\copart$ forms a $k$-coalition with $Y_1$ or $Y_2$. Hence, $\copart$ is a $k$-coalition partition and we have 
        \begin{equation*}
             |\copart| \geq |V\setminus(Y_1 \cup Y_2)| + 2 \geq \sum_{i=1}^r s_i - 2k - 2\left\lceil \frac{k}{r-1} \right\rceil +3,
         \end{equation*}
  which proves \eqref{Multipartite Case 6}.
\end{proof}
\begin{figure} [htpb]
    \centering
        \scalebox{0.85}{
\begin{tikzpicture}[
  every node/.style={circle, draw, fill=white, inner sep=1.5pt}, 
  scale=1
]

\foreach \i in {1,2,3,4,5} {
  \node  (s1\i) at (-4, 3 - 0.75*\i) {};
  \node[draw=none, fill = none] at (-4, 3.2 - 0.75*\i) {$v_{1,\i}$};
  \node (s2\i) at (0, 3 - 0.75*\i) {};
    \node[draw=none, fill = none] at (0, 3.2 - 0.75*\i) {$v_{2,\i}$};
  \node (s3\i) at (4, 3 - 0.75*\i) {};
    \node[draw=none, fill = none] at (4, 3.2 - 0.75*\i) {$v_{3, \i}$};
}

\node[draw=none, fill=none] at (-4, -1.23) {$\vdots$};
\node[draw=none, fill=none] at (0, -1.23)  {$\vdots$};
\node[draw=none, fill=none] at (4, -1.23)  {$\vdots$};

\node[draw=none, fill = none] at (-4, -2) {$v_{1,s_1}$};
\node[draw=none, fill = none] at (0, -2) {$v_{2,s_2}$};
\node[draw=none, fill = none] at (4, -2){$v_{3,s_3}$};
\node (s17)  at (-4, -1.8) {};
\node (s27)  at (0, -1.8) {};
\node (s37) at (4, -1.8) {};


\begin{scope}[on background layer]
\draw[draw = black] (-4,0.2) ellipse (1 and 2.8);
\draw[draw = black] (0,0.2) ellipse (1 and 2.8);
\draw[draw = black] (4,0.2) ellipse (1 and 2.8);
\end{scope}

\begin{scope}[on background layer]
\node[ellipse, draw = Cerulean, fill = Cerulean!30, minimum width=1cm, minimum height=1.5cm, fit=(s11)(s12)] {};
\node[ellipse, draw = Cerulean, fill = Cerulean!30, minimum width=1cm, minimum height=1.5cm, fit=(s21)(s22)] {};
\node[ellipse, draw =  Cerulean, fill = Cerulean!30, minimum width=0.9cm, minimum height=1cm, fit=(s31)] {};
\end{scope}

    

\begin{scope}[on background layer]
\node[ellipse, draw = OliveGreen, fill = LimeGreen!30, minimum width=0.9cm, minimum height=1cm, fit=(s13) ] {};
\node[ellipse, draw = OliveGreen, fill = LimeGreen!30, minimum width=1cm, minimum height=1.5cm, fit=(s23)(s24)] {};
\node[ellipse, draw = OliveGreen, fill = LimeGreen!30, minimum width=0.9cm, minimum height=1cm, fit=(s32) ] {};
\node[ellipse, draw = OliveGreen, fill = LimeGreen!30, minimum width=1cm, minimum height=1.5cm, fit=(s33)(s34)] {};
\end{scope}

\node[draw=none, fill=none] at (-3.2, 3) {\( S_1 \)};
\node[draw=none, fill=none] at (0.8, 3)  {\( S_2 \)};
\node[draw=none, fill=none] at (4.8, 3)  {\( S_3 \)};

\end{tikzpicture}}
\caption{The complete $3$-partite graph $K_{s_1, s_2, s_3}$ with the $4$-coalition partition with the lower bound construction in \eqref{Multipartite Case 6}. The set $Y_1$ is shown in blue, the set $Y_2$ is shown in green, and all remaining vertices form singleton sets.}
    \label{fig: Multipartite Case 6}
\end{figure}
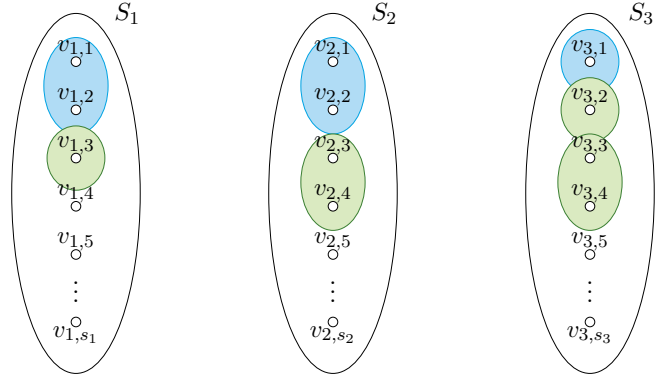

\end{document}